\documentclass{amsart}

\RequirePackage{amsmath}
\RequirePackage{amsthm}
\RequirePackage{amsfonts}
\RequirePackage{amssymb}
\RequirePackage{multirow}
\RequirePackage{multicol}
\RequirePackage{graphicx}
\RequirePackage{float}
\RequirePackage{enumerate} 
\RequirePackage{cancel}
\RequirePackage{appendix}
\RequirePackage{calrsfs}
\RequirePackage{caption}
\RequirePackage{xcolor}
\RequirePackage{hyperref}

\RequirePackage{tikz}
\usetikzlibrary{babel}
\usetikzlibrary{shapes,arrows}
\tikzstyle{block} = [draw, fill=blue!20, rectangle, 
    minimum height=3em, minimum width=3em]
\tikzstyle{sum} = [draw, fill=blue!20, circle, node distance=1cm]
\tikzstyle{input} = [coordinate]
\tikzstyle{output} = [coordinate]

\newtheorem{theorem}{Theorem}[section]
\newtheorem{proposition}[theorem]{Proposition}
\newtheorem{corollary}[theorem]{Corollary}
\newtheorem{lemma}[theorem]{Lemma}

\theoremstyle{definition}

\newtheorem{definition}[theorem]{Definition}

\newtheorem{remark}[theorem]{Remark}

\newcommand{\C}{\mathbb{C}}

\newcommand{\R}{\mathbb{R}}
\newcommand{\D}{\mathbb{D}}
\newcommand{\N}{\mathbb{N}}
\newcommand{\Z}{\mathbb{Z}}

\newcommand{\abs}[1]{\left| #1 \right|}

\title[The Bergman and the growth numbers of domains]{The Bergman and the growth numbers of domains and hyperbolic geometry}
\date{\today}
\thanks{This research was supported in part by Ministerio de Innovaci\'on y Ciencia, Spain, project PID2022-136320NB-I00. The second author was supported by Ministerio de Universidades, Spain, through the action Ayuda del Programa de Formaci\'on de Profesorado Universitario, reference FPU21/00258. The authors thanks IMUS-Mar\'ia de Maeztu grant CEX2024-001517-M - Apoyo a Unidades de Excelencia Mar\'ia de Maeztu for supporting this research, funded by MICIU/AEI/ 10.13039/501100011033}
\author[M. D. Contreras]{Manuel D. Contreras}
\address{Departamento de Matem\'atica Aplicada II and IMUS, Escuela T\'ecnica Superior de Ingenier\'ia, Universidad de Sevilla, Camino de los Descubrimientos, s/n 41092, Sevilla, Spain}
\email{contreras@us.es}
\author[F. J. Cruz-Zamorano]{Francisco J. Cruz-Zamorano}
\address{Departamento de An\'alisis Matem\'atico, Facultad de Ciencias, Universidad de La Laguna, Avenida Astrof\'isico Francisco S\'anchez S/N, 38206 San Crist\'obal de La Laguna, Santa Cruz de Tenerife, Spain}
\email{fcruzzam@ull.edu.es}
\author[L. Rodr\'iguez-Piazza]{Luis Rodr\'iguez-Piazza}
\address{Departamento de An\'alisis Matem\'atico and IMUS, Facultad de Matem\'aticas, Universidad de Sevilla, Calle Tarfia, s/n 41012 Sevilla, Spain}
\email{piazza@us.es}

\DeclareMathOperator{\h}{h}
\let\b\relax
\DeclareMathOperator{\b}{b}
\DeclareMathOperator{\bloch}{Bloch}
\DeclareMathOperator{\g}{g}

\newcommand{\fstbox}[2]{
\draw (#1,---#2) -- (---#2,---#2);
\draw (---#2,---#2) -- (---#2,#2);
\draw (---#2,#2) -- (#1,#2);
}

\newcommand{\xbox}[4]{
\draw (#1,#3) -- (#1,#2*0.5-#1*0.5);
\draw (#1,#2*0.5-#1*0.5) -- (#2,#2*0.5-#1*0.5);
\draw (#2,#2*0.5-#1*0.5) -- (#2,#4);
\draw (#1,---#3) -- (#1,#1*0.5-#2*0.5);
\draw (#1,#1*0.5-#2*0.5) -- (#2,#1*0.5-#2*0.5);
\draw (#2,#1*0.5-#2*0.5) -- (#2,---#4);
}

\newcommand{\xbanda}[3]{
\draw (#1,#3) -- (#2,#3);
\draw (#1,---#3) -- (#2,---#3);
}

\begin{document}

\begin{abstract}
In this article we completely characterize the Bergman number of general domains. Namely, we prove that the Bergman number of a hyperbolic unbounded domain can be calculated in terms of the asymptotic behavior of its hyperbolic metric near  infinity. To obtain the proof of this result we first work in the setting of growth spaces, defining the growth number of a domain. Later, we prove an equality relating the Bergman number and the growth number of domains. We also provide examples of domains with prescribed Bergman number which have zero Hardy number, solving a question posed previously by Betsakos and Cruz-Zamorano. At the end, a similar idea is treated for the case of Bloch-type spaces.
\end{abstract}

\maketitle

\section{Introduction}
Let $X$ be a space of holomorphic functions on the unit disk $\D$, and let $\Omega \subset \C$ be a domain (i.e., an open and connected subset of the complex plane). A long-standing question involves deciding whether all holomorphic maps on $\D$ with values on $\Omega$ lie in $X$ (i.e., whether $\mathrm{Hol}(\D,\Omega) \subset X$). This problem has been fruitfully studied over the last century, particularizing in the case when $X$ is the Nevanlinna class (see \cite[Theorems 5.1.1 and 5.4.2, pp. 209 and 211]{Nevanlinna}), the Smirnov class \cite{AhernCohn}, the space of holomorphic maps with bounded mean oscillation BMOA \cite{HaymanPommerenke}, the Bloch space (see \cite{Cima}), or a Hardy space \cite{CCZKRP,Essen,Hansen,K-HypDist,K-Comb1,K-Comb2,KimSugawa}. For these cases, the known characterizations usually depend on different geometric or potential-theoretic attributes of the domain $\Omega$.

In this  paper we will focus on the case of Bergman spaces, which has been recently studied in \cite{BCZ,Karafyllia,KK}. More precisely, we will discuss how to characterize the Bergman number of a domain using the hyperbolic distance.

Our ideas start with the so-called growth spaces. A first exhaustive study about these spaces first appeared in \cite{SW}. The state of the art about this topic can be found in the survey \cite{Bonet-Survey}. A growth space $\mathrm{H}^{\infty}_v$ is the collection of all holomorphic maps $f \colon \D \to \C$ such that
$$\sup_{z \in \D}\abs{f(z)}v(z) < +\infty,$$
where $v \colon \D \to [0,+\infty)$ is a weight (i.e., a non-negative function). In the case that $\Omega \subset \C$ is a hyperbolic domain, we will characterize whether $\mathrm{Hol}(\D,\Omega) \subset \mathrm{H}^{\infty}_v$ for a class of typical weights $v$ (the terminology comes from \cite[p. 102]{BDLT}); see Theorem \ref{thm:univconmap-Hinftyv}.

We will use a relation between the Bergman spaces and the growth spaces $\mathrm{H}^{\infty}_{v_{\alpha}}$, where $v_{\alpha}(z) = (1-\abs{z})^{\alpha}$, $z \in \D$, for some $\alpha \geq 0$. Notice that the usual Banach space $\mathrm{H}^{\infty}$ of bounded holomorphic functions in $\D$ coincides with $\mathrm{H}^{\infty}_{v_0}$.

It is straightforward that $\mathrm{H}^{\infty}_{v_{\alpha}} \subset \mathrm{H}^{\infty}_{v_{\beta}}$ if and only if $\alpha \leq \beta$. That is, the growth spaces $\mathrm{H}^{\infty}_{v_{\alpha}}$, $\alpha \geq 0$, form an increasing chain of spaces. Motivated by the definition of the Hardy number of a domain (see \cite[p. 292]{KimSugawa}, for instance), we define the growth number of a domain $\Omega \subset \C$ as
$$\g(\Omega) = \inf (\{\alpha > 0 : \mathrm{Hol}(\D,\Omega) \subset \mathrm{H}^{\infty}_{v_{\alpha}}\}),$$
where we use the convention that $\inf(\emptyset) = +\infty$.

As usual, the growth number of domains with a simple geometry (such as angular sectors or strips) is easy to calculate (more information follows in Section \ref{sec:growth-number}). The interesting cases are covered by the following complete characterization in terms of the hyperbolic distance (some notations and complementary information can be found in Section \ref{sec:preliminaries}).
\begin{theorem}
\label{thm:growth-number}
Let $\Omega \subset \C$ be an unbounded hyperbolic domain with $0 \in \Omega$. Then
$$\g(\Omega) = \limsup_{\substack{w \to \infty \\ w \in \Omega}}\dfrac{\log(\abs{w})}{d_{\Omega}(0,w)} = \limsup_{R \to +\infty}\dfrac{\log(R)}{d_{\Omega}(0,F_R)},$$
where $F_R = \{w \in \Omega : \abs{w} = R\}$, $R > 0$.
\end{theorem}
Notice that $\g(\Omega) = \g(a\Omega+b)$ for all $a,b \in \C$, $a \neq 0$. Therefore, the assumption that $0 \in \Omega$ in the previous result is not crucial. In fact, we will consider such normalization in the rest of the paper.

The latter result is intimately linked to the case of Bergman spaces. To elaborate on this, let us recall that, for $p > 0$ and $\alpha > -1$, the (weighted) Bergman space $A^p_{\alpha}$ stands for the set of all holomorphic maps $f$ on $\D$ such that
$$\int_{\D}(1-\abs{z}^2)^{\alpha}\abs{f(z)}^pdA(z) < +\infty.$$
Notice that $A^p = A^p_0$ is the classical Bergman space.

Building on \cite{KK}, Karafyllia introduced the Bergman number of a domain in \cite{Karafyllia}. Namely, for a domain $\Omega \subset \C$, she defined
$$\b(\Omega) = \inf(\{\b(f) : f \in \mathrm{Hol}(\D,\Omega)\}),$$
where
$$\b(f) = \sup\left(\left\lbrace \dfrac{p}{\alpha+2} : p > 0, \, \alpha > -1, \, f \in A^p_{\alpha} \right\rbrace\right)$$
and again we use the convention that $\inf(\emptyset) = +\infty$ and $\sup(\emptyset) = 0$. 


Karafyllia and Karamanlis \cite{KK} proved that the Bergman number and the Hardy number of simply connected domains coincide. Such equality does not hold for general domains, as seen in \cite{BCZ}. Our main result deals with a similar relation with the growth number.

\begin{theorem}
\label{thm:bergman-growth}
Let $\Omega \subset \C$ be a domain. Then,
\begin{enumerate}[\hspace{0.5cm}\rm(a)]
\item $\g(\Omega) = 0$ if and only if $\b(\Omega) = +\infty$.
\item $\g(\Omega) = +\infty$ if and only if $\b(\Omega) = 0$.
\item If $0 < \g(\Omega) < +\infty$, then $\g(\Omega)\b(\Omega) = 1$.
\end{enumerate}
In particular, if $\Omega$ is hyperbolic, unbounded, and $0 \in \Omega$, then
$$\b(\Omega) = \liminf_{R \to +\infty}\dfrac{d_{\Omega}(0,F_R)}{\log(R)},$$
where $F_R = \{w \in \Omega : \abs{w} = R\}$, $R > 0$.
\end{theorem}

We recall that bounded domains have infinite Bergman number, while non-hyperbolic domains have zero Bergman number \cite[Lemma 2.2]{Karafyllia}. Then, the latter result yields a complete characterization of the Bergman number of a general domain in terms of the hyperbolic distance. We should mention that this characterization was known for simply connected domains \cite[Theorem 1.3.(3)]{KK}. However, our arguments are completely different. In fact, as far as we know, this is the first characterization of the Bergman number of a general domain.

As an application of our results, we will also derive the following:
\begin{theorem}
\label{thm:examples}
\begin{enumerate}[\rm(a)]
\item For every $\alpha \in [0,+\infty]$, there exists a domain $\Omega \subset \C$ such that $\g(\Omega) = \alpha$.
\item For every $q \in [0,+\infty]$, there exists a domain $\Omega \subset \C$ such that $\b(\Omega) = q$.
\end{enumerate}
\end{theorem}
The proof of the latter result involves the construction of explicit examples. In particular, for the cases in which $\b(\Omega) \in (0,1/2)$, the complement of each of the domains we build is countable. This has some consequences with respect to the Hardy number of these domains, which will lead us to answer a question that was posed in \cite{BCZ}.

The organization of the article is as follows: first, we give the needed preliminaries in Section \ref{sec:preliminaries}. This mainly concerns the hyperbolic domains, their universal covering maps, and their hyperbolic distance. After that, we discuss the growth spaces in Section \ref{sec:growth-spaces} and the growth number of a domain in Section \ref{sec:growth-number}, proving Theorem \ref{thm:growth-number}. We continue the discussion by proving Theorem \ref{thm:examples}.(a) in Section \ref{sec:examples}. After that, we elaborate on the Bergman number of a domain in Section \ref{sec:bergman}, and we deduce Theorem \ref{thm:bergman-growth} using the inclusions among the Bergman spaces and the growth spaces, as well as Theorem \ref{thm:examples}.(b). We state some consequences of our work in Section \ref{sec:questions}. Namely, we find domains with prefixed Bergman number but zero Hardy number. Lastly, in Section \ref{sec:bloch}, we introduce the Bloch number of a domain and use it to complement a classical theorem for the Bloch space. We also derive some examples in Section \ref{sec:final-example}.

\textit{Acknowledgement}. The authors want to thank Artur Nicolau for the fruitful discussions during his visit to the Instituto de Matem\'aticas de la Universidad de Sevilla, which led to the content in Section \ref{sec:final-example}.

\section{Preliminaries}
\label{sec:preliminaries}

\subsection{Hyperbolic domains and their universal covering maps}
\label{subsec:covering}
A domain $\Omega \subset \C$ is called hyperbolic if its complement $\C \setminus \Omega$ contains at least two points (i.e., $\Omega \neq \C$ and there exists no $w \in \C$ such that $\Omega = \C \setminus \{w\}$). These domains enjoy several special features when seen as Riemann surfaces. For instance, every hyperbolic domain is covered by the unit disk, which means that there exists a universal covering map $f_{\Omega} \colon \D \to \Omega$ (see \cite[Theorem 16.5.1]{ConwayII} for a construction). This holomorphic map has the following properties:

(I) For every $f \in \mathrm{Hol}(\D, \Omega)$ there exists a holomorphic map $g \colon \D \to \D$ such that $f = f_{\Omega} \circ g$. This fact is known as the lifting property.

(II) Consider the group $\Gamma = \{\gamma \in \mathrm{Aut}(\D) : f_{\Omega} \circ \gamma = f_{\Omega}\}$. For $z \in \D$ and $w \in \Omega$ such that $f_{\Omega}(z) = w$, we have that
$$f^{-1}_{\Omega}(w) = \{\gamma(z) : \gamma \in \Gamma\}.$$

We refer to \cite[Section 1.6]{Abate} and \cite[Sections 1 and 2]{Milnor} for more basic information.

\begin{remark}
Let $\Omega \subset \C$ be a hyperbolic domain. Every universal covering map of $\Omega$ is of the form $g_{\Omega} = f_{\Omega} \circ g$, where $f_{\Omega}$ is a universal covering map of $\Omega$ and $g \in \mathrm{Aut}(\D)$. Due to this fact, given $z \in \D$ and $w \in \Omega$, there is a universal covering map satisfying $f_{\Omega}(z) = w$.
\end{remark}

\begin{remark}
\label{remark:simply-connected}
A hyperbolic domain $\Omega \subsetneq \C$ is simply connected if and only if $\Gamma = \{\mathrm{Id}_{\D}\}$, in which case $f_{\Omega}$ is a Riemann map of $\Omega$.
\end{remark}

\subsection{Hyperbolic distance}
The hyperbolic distance of a hyperbolic domain is a conformal invariant. We refer to \cite{BeardonMinda} for a comprehensive monograph about this topic. Usually (see, for instance, \cite{Milnor}), it is defined as the distance associated with a conformally invariant metric. For the purpose of our research, we introduce it in the following equivalent form:

If $z_1,z_2 \in \D$, the hyperbolic distance between $z_1$ and $z_2$ in $\D$ is given by
\begin{equation}
\label{eq:hypdist-D}
d_{\D}(z_1,z_2) = \log\left(\dfrac{1+\rho_{\D}(z_1,z_2)}{1-\rho_{\D}(z_1,z_2)}\right), \qquad \text{where} \qquad \rho_{\D}(z_1,z_2) = \abs{\dfrac{z_1-z_2}{1-z_1\overline{z_2}}}.
\end{equation}
The conformal invariance of this definition comes from the Schwarz-Pick Lemma.

In the case of a general hyperbolic domain $\Omega \subset \C$, consider a universal covering map $f_{\Omega}$. Let $w_1,w_2 \in \Omega$, and choose $z_k \in \D$ so that $f_{\Omega}(z_k) = w_k$ for $k \in \{1,2\}$. Then, we define the hyperbolic distance between $w_1$ and $w_2$ in $\Omega$ as 
\begin{align}
\label{eq:domega-inf}
d_{\Omega}(w_1,w_2) := &\inf(\{d_{\D}(a,b) : a,b \in \D, \, f_{\Omega}(a) = w_1, \, f_{\Omega}(b) = w_2\}) \\
= & \inf_{\substack{z \in \D \\ f_{\Omega}(z) = w_2}}d_{\D}(z_1,z) = \inf_{\gamma \in \Gamma}d_{\D}(z_1,\gamma(z_2)), \notag
\end{align}
where $\Gamma$ is the group of automorphism associated with $\Omega$. We will very often use that the infimum in the latter definition is attained (see \cite[Proposition 1.7.3.(iv)]{Abate}).

We also recall that the hyperbolic distance has a monotonicity property with respect to the domain \cite[Theorem 10.5]{BeardonMinda}. Namely, if $\Omega_1 \subset \Omega_2 \subset \C$ are hyperbolic domains, then
\begin{equation}
\label{eq:d-omega-monotonicity}
d_{\Omega_1}(w_1,w_2) \geq d_{\Omega_2}(w_1,w_2), \qquad w_1,w_2 \in \Omega_1.
\end{equation}

\subsection{The Dirichlet polygon}
\label{subsec:dirichlet}
This subsection is devoted to recall some of the tools which can be found in \cite[Chapter 9]{Beardon}. 

Let $\Omega \subset \C$ be a hyperbolic multiply connected domain, and choose $z \in \D$. The Dirichlet polygon of $\Omega$ with center $z$ is given by
$$\mathcal{D} = \{w \in \D : d_{\D}(z,w) < d_{\D}(z,\gamma(w)) \text{ for all } \gamma \in \Gamma \setminus \{\mathrm{Id}_{\D}\}\},$$
where $\Gamma$ is the group of automorphisms associated with $\Omega$.

Dirichlet polygons are (convex) fundamental domains \cite[Theorem 9.4.2]{Beardon}, meaning that they have the following properties:

(a) there exists a set $\mathcal{D} \subset A \subset \overline{\mathcal{D}} \cap \D$ such that $f_{\Omega}$ acts as a bijection from $A$ to $\Omega$, and

(b) the hyperbolic area of $\partial \mathcal{D} \cap \D$ is null, meaning that
$$\int_{\partial \mathcal{D} \cap \D}\dfrac{dA(z)}{(1-\abs{z}^2)^2} = 0.$$
In particular, $A(\partial \mathcal{D}) = 0$.

Therefore, if $\Gamma$ is the group of automorphisms associated with $\Omega$, then we see that
$$\bigcup_{\gamma \in \Gamma} \gamma(\mathcal{D})$$
is a pairwise disjoint union of fundamental domains which is dense in $\D$.

In the case that $\Omega \subset \C$ is a simply connected domain, by Remark \ref{remark:simply-connected} we clearly see that $f_{\Omega} \colon \D \to \Omega$ is a bijection. In that case, we define $\mathcal{D} =  \D$ as the Dirichlet polygon of $\Omega$ (for any center $z \in \D$).

\subsection{Carath\'eodory's Theorem}
A well-known result of Carath\'eodory aims to characterize when the conformal maps of a sequence of simply connected domains converge in terms of their geometry \cite[Theorem 3.5.8]{BCDM}. This result has been generalized by Hejhal \cite{Hejhal} to the case of universal covering maps of hyperbolic domains. In order to understand the result, we need to introduce some notation.

\begin{definition}
Let $z \in \C$ and let $\{\Omega_n\}$ be a sequence of domains such that $z \in \Omega_n$ for all $n \in \N$. The kernel of $\{\Omega_n\}$ with respect to $z$ is $G \cup \{z\}$, where $G \subset \C$ is the (possibly empty) set of all points $w \in \C$ such that there exists an open connected set $D \subset \C$ with the property that $\{z,w\} \subset D$ and $D \subset \Omega_n$ for all but a finite number of $n \in \N$.

We say that $\{\Omega_n\}$ converges to its kernel with respect to
$z$ if every subsequence of $\{\Omega_n\}$ has the same kernel with respect to $z$.
\end{definition}

This definition comes from \cite[p. 28]{PommerenkeUnivFun}. It is equivalent to the one originally given by Hejhal in \cite[p. 7]{Hejhal}.

The result of Hejhal links the convergence of universal covering maps with the convergence of their associated hyperbolic domains.
\begin{theorem}
\label{thm:Hejhal}
{\normalfont \cite[Theorem 1]{Hejhal}}
Let $z \in \C$ and let $\{\Omega_n\}$ be a sequence of hyperbolic domains such that $z \in \Omega_n$ for all $n \in \N$. Denote by $f_{\Omega_n}$ the universal covering map of $\Omega_n$ satisfying that $f_{\Omega_n}(0) = z$ and $f'_{\Omega_n}(0) > 0$. The following facts are equivalent:
\begin{enumerate}[\hspace{0.5cm}\rm(a)]
\item There exists a non-constant holomorphic map $f \colon \D \to \C$ such that $f_{\Omega_n}$ converges to $f$ locally uniformly, as $n \to \infty$.
\item The kernel of $\{\Omega_n\}$ with respect to $z$ is a hyperbolic domain and $\{\Omega_n\}$ converges to its kernel with respect to $z$.
\end{enumerate}
If this is the case, then $f$ is the universal covering map of the kernel of $\{\Omega_n\}$ with respect to $z$, satisfying $f(0) = z$ and $f'(0) > 0$.
\end{theorem}

For a hyperbolic domain $\Omega \subset \C$, recall that the hyperbolic distance can be introduced as the one associated with the metric given by
\begin{equation}
\label{eq:metric}
\lambda_{\Omega}(f_{\Omega}(z)) = \dfrac{\lambda_{\D}(z)}{\abs{f'_{\Omega}(z)}}, \qquad \lambda_{\D}(z) = \dfrac{2}{1-\abs{z}^2}, \qquad z \in \D,
\end{equation}
where the definition does not depend on the chosen universal covering map $f_{\Omega}$. In this way, we can compare $d_{\Omega_n}$ and $d_{\Omega}$, as $n \to \infty$. In order to do so, notice that if $w$ is in the kernel of $\{\Omega_n\}$, then there exists $N \in \N$ such that $w \in \Omega_n$ for all $n \geq N$. With this in mind, we state the following consequence for which we were not able to find a proper reference:
\begin{corollary}
\label{cor:distance-convergence}
Let $z \in \C$ and let $\{\Omega_n\}$ be a sequence of hyperbolic domains such that $z \in \Omega_n$ for all $n \in \N$. Assume that the kernel $\Omega \subset \C$ of $\{\Omega_n\}$ with respect to $z$ is a hyperbolic domain. If $\{\Omega_n\}$ converges to $\Omega$ with respect to $z$, then
\begin{equation}
\label{eq:to-prove}
\lim_{n \to \infty}d_{\Omega_n}(a,b) = d_{\Omega}(a,b), \qquad a,b \in \Omega.
\end{equation}
Moreover, if $\{b_n\}$ is a sequence in $\Omega$ with $b_n \to b \in \Omega$, as $n \to \infty$, then $b_n \in \Omega_n$ for $n$ large enough and
$$\lim_{n \to \infty}d_{\Omega_n}(a,b_n) = d_{\Omega}(a,b), \qquad a \in \Omega.$$
\begin{proof}
To fix the notation, let $f_n = f_{\Omega_n}$ be the universal covering map of $\Omega_n$ with $f_n(0) = z$ and $f'_n(0) > 0$. Likewise, let $f = f_{\Omega}$ be the universal covering map of $\Omega$ with $f(0) = z$ and $f'(0) > 0$.

Let us first prove \eqref{eq:to-prove}. To do so, let $A,B $ be two points in $ \D$ such that $f(A) = a$, $f(B) = b$, and $d_{\Omega}(a,b) = d_{\D}(A,B)$; see \eqref{eq:domega-inf} and the subsequent comments. By Theorem \ref{thm:Hejhal}, notice that $f_n(A) \to a$ and $f_n(B) \to b$, as $n \to \infty$. Then, applying the triangle inequality and the Schwarz-Pick Lemma we get
\begin{align*}
d_{\Omega_n}(a,b) & \leq d_{\Omega_n}(a,f_n(A)) + d_{\Omega_n}(f_n(A),f_n(B)) + d_{\Omega_n}(f_n(B),b) \\
& \leq d_{\Omega_n}(a,f_n(A)) + d_{\D}(A,B) + d_{\Omega_n}(f_n(B),b).
\end{align*}
Since $a \in \Omega$, there exist $\varepsilon > 0$ and $N \in \N$ so that $\Delta := \{z \in \C : \abs{z-a} < \varepsilon\} \subset \Omega_n$ for all $n \geq N$. Since $f_n(A) \to a$, as $n \to \infty$, there must exists $M \in \N$ so that $f_n(A) \in \Delta $ for all $n \geq M$. Using \eqref{eq:d-omega-monotonicity}, if $n \geq M$, then
$$d_{\Omega_n}(a,f_n(A)) \leq d_{\Delta} (a,f_n(A)) \to 0, \qquad \text{as } n \to \infty.$$
Arguing similarly, we also have $d_{\Omega_n}(b,f_n(B)) \to 0$, as $n \to \infty$. Thus,
\begin{equation}
\label{eq:d1}
\limsup_{n \to \infty} d_{\Omega_n}(a,b) \leq d_{\D}(A,B) = d_{\Omega}(a,b),
\end{equation}

To continue with the proof of \eqref{eq:to-prove}, choose a subsequence such that
$$\lim_{k \to \infty}d_{\Omega_{n_k}}(a,b) = \liminf_{n \to \infty}d_{\Omega_n}(a,b).$$
By Hurwitz Lemma, let $A_{n_k}$ be a point in $\D$ such that $f_{n_k}(A_{n_k}) = a$ and $A_{n_k} \to A$, as $k \to \infty$. For each $k \in \N$ we can also find a point $B_{n_k} \in \D$ such that $f_{n_k}(B_{n_k}) = b$ and $d_{\D}(A_{n_k},B_{n_k}) = d_{\Omega_{n_k}}(a,b)$. Using a further subsequence if needed, we suppose that there exists a point $\widetilde{B} \in \overline{\D}$ such that $B_{n_k} \to \widetilde{B}$, as $k \to \infty$. If $\widetilde{B} \in \partial\D$, then $d_{\D}(A_{n_k},B_{n_k}) \to +\infty$, but
$$d_{\D}(A_{n_k},B_{n_k}) = d_{\Omega_{n_k}}(a,b) \to \liminf_{n \to \infty}d_{\Omega_n}(a,b) < +\infty, \qquad \text{as } n \to \infty.$$
For these reasons, we see that $\widetilde{B} \in \D$ and $f(\widetilde{B}) = b$. Therefore, by \eqref{eq:domega-inf}, we notice that
$$d_{\Omega}(a,b) \leq d_{\D}(A,\widetilde{B}) \leq d_{\D}(A,A_{n_k}) + d_{\D}(A_{n_k},B_{n_k}) + d_{\D}(B_{n_k},\widetilde{B}).$$
Then,
\begin{equation}
\label{eq:d2}
d_{\Omega}(a,b) \leq \lim_{k \to \infty}d_{\D}(A_{n_k},B_{n_k}) = \lim_{k \to \infty}d_{\Omega_{n_k}}(a,b) = \liminf_{n \to \infty}d_{\Omega_{n}}(a,b),
\end{equation}
where we have used that $d_{\D}$ is locally equivalent to the Euclidean distance.
So, the convergence in \eqref{eq:to-prove} follows from the combination of \eqref{eq:d1} and \eqref{eq:d2}.

To complete the proof, let $\{b_n\}$ be a sequence in $\Omega$ with $b_n \to b \in \Omega$, as $n \to \infty$. As before, since $b \in \Omega$, we can choose $\varepsilon > 0$ and $N \in \N$ so that $\Delta := \{z \in \C : \abs{z-b} < \varepsilon\} \subset \Omega_n$ for all $n \geq N$. For sufficiently large $n$, we have that
$$d_{\Omega_n}(a,b_n) \leq d_{\Omega_n}(a,b) + d_{\Omega_n}(b_n,b).$$
As before,
$$d_{\Omega_n}(b_n,b) \leq d_{\Delta}(b_n,b) \to 0, \qquad \text{as } n \to \infty.$$
Then, by \eqref{eq:to-prove},
$$\limsup_{n \to \infty}d_{\Omega_n}(a,b_n) \leq d_{\Omega}(a,b).$$

Similarly, we have
$$d_{\Omega_n}(a,b) \leq d_{\Omega_n}(a,b_n) + d_{\Omega_n}(b_n,b),$$
from which it follows that
$$d_{\Omega}(a,b) \leq \liminf_{n \to \infty}d_{\Omega_n}(a,b_n).$$
This proves the result.
\end{proof}
\end{corollary}

\subsection{An estimate for the hyperbolic metric}
\label{sec:BeardonPommerenke}
During the next sections we will need to estimate the hyperbolic distance on domains with a precise geometry. To this extent, we first recall the following well-known equivalence between the hyperbolic density of a simply connected domain $\Omega \subsetneq \C$ at $z \in \Omega$ and the distance of $z$ to the boundary, that is, $\mathrm{dist}(z,\partial\Omega) = \inf\{\abs{z-w} : w \in \partial \Omega\}$.
\begin{theorem}
\label{thm:estimate-density-simpcon}
{\normalfont \cite[Theorems 8.2 and 8.6]{BeardonMinda}}
Let $\Omega \subsetneq \C$ be a simply connected domain. Then,
$$\dfrac{1}{2\mathrm{dist}(z,\partial\Omega)} \leq \lambda_{\Omega}(z) \leq \dfrac{1}{\mathrm{dist}(z,\partial\Omega)}, \qquad z \in \Omega.$$
\end{theorem}
The above result does not hold for general hyperbolic domains $\Omega \subset \C$. For instance, using \cite[Eq. (1.22)]{Ahlfors}, we see that
$$\lim_{x \to 0}x\lambda_{\C \setminus \{0,1\}}(x) = 0.$$

Beardon and Pommerenke have proved a refinement of Theorem \ref{thm:estimate-density-simpcon}. To introduce it, we first give some notation. For a hyperbolic domain $\Omega \subset \C$ and $z \in \Omega$, we define
$$\beta_{\Omega}(z) = \inf\left\lbrace\abs{\log\abs{\dfrac{z-a}{b-a}}} : a,b \in \partial\Omega, \, \abs{z-a} = \mathrm{dist}(z,\partial\Omega)\right\rbrace.$$

\begin{theorem}
\label{thm:BeardonPommerenke}
{\normalfont \cite[Theorem 1]{BeardonPommerenke}}
Let $\Omega \subset \C$ be a hyperbolic domain and let $z \in \Omega$. It holds that
$$\dfrac{1}{2\sqrt{2}} \leq \lambda_{\Omega}(z)\mathrm{dist}(z,\partial\Omega)(\kappa+\beta_{\Omega}(z)) \leq \kappa+\dfrac{\pi}{4},$$
where $\kappa = 4+\log(3+2\sqrt{2})$.
\end{theorem}

\section{Growth spaces}
\label{sec:growth-spaces}
For a large class of weights $v$, this section aims to characterize the hyperbolic domains $\Omega \subset \C$ for which $\mathrm{Hol}(\D,\Omega) \subset \mathrm{H}^{\infty}_v$. We will later deduce Theorem \ref{thm:growth-number} by considering the weights $v = v_{\alpha}$, $\alpha > 0$, in such a characterization.

In order to specify the class of weights we will consider in this work, we introduce the following definitions.
\begin{definition}
A weight $v \colon \D \to [0,+\infty)$ is called typical if
\begin{enumerate}[\hspace{0.5cm}(i)]
\item it is positive, meaning that $v(z) > 0 $ for all $z \in \D$;
\item it is radial, meaning that $v(z) = v(\abs{z})$ for all $z \in \D$;
\item it is continuous and non-increasing with respect to $\abs{z}$;
\item it converges to zero on the boundary, meaning that $\lim_{\abs{z} \to 1^-}v(z) = 0$.
\end{enumerate}
\end{definition}
Notice that $v_{\alpha}$, $\alpha > 0$, is always a typical weight. Also, notice that if a positive radial continuous and non-increasing weight $v$ satisfies that
$$\lim_{\abs{z} \to 1^-}v(z) > 0,$$
then $\mathrm{H}^{\infty}_v = \mathrm{H}^{\infty}$.

For every weight $v$, one may construct the associated weight given by
\begin{equation*}
\label{eq:associated-weight}
\widetilde{v}(z) = \inf_{f \in \mathbb{B}(\mathrm{H}^{\infty}_v)}\dfrac{1}{\abs{f(z)}}, \qquad z \in \D,
\end{equation*}
where $\mathbb{B}(\mathrm{H}^{\infty}_v) = \{f \in \mathrm{H}^{\infty}_v : \abs{f(z)}v(z) \leq 1 \text{ for all } z \in \D\}$. Notice that $\widetilde{v}(z) \geq v(z)$ for all $z \in \D$. Indeed, $\mathrm{H}^{\infty}_v = \mathrm{H}^{\infty}_{\widetilde{v}}$. In fact, with their usual norms, they are isometric \cite[Proposition 1.3]{BDLT}.

The notion of the associated weight is useful, for instance, to characterize the boundedness of the composition operators on $\mathrm{H}^{\infty}_v$. Recall that, for a holomorphic map $g \colon \D \to \D$, the composition operator $C_g \colon \mathrm{Hol}(\D,\C) \to \mathrm{Hol}(\D,\C) $ is given by $C_g(f) = f \circ g$, where $f \in \mathrm{Hol}(\D,\C)$.
Thanks to the Closed Graph Theorem, for a Banach space of holomorphic functions $X \subset \mathrm{Hol}(\D,\C)$ such that the convergence in norm implies the uniform convergence in compact sets, $C_g$ is bounded if and only if $C_g(X) \subset X$.
\begin{theorem}
\label{thm:continuity-Cg-Hinftyv}
{\normalfont \cite[Theorem 2.3]{BDLT}}
Let $v$ be a typical weight. The following facts are equivalent:
\begin{enumerate}[\hspace{0.5cm} \rm (a)]
\item All composition operators $C_g \colon \mathrm{H}^{\infty}_v \to \mathrm{H}^{\infty}_v$, $g \in \mathrm{Hol}(\D,\D)$, are bounded.
\item The associated weight $\widetilde{v}$ satisfies
\begin{equation}
\label{eq:continuity-wtv}
\sup_{n \in \N}\dfrac{\widetilde{v}(1-2^{-n})}{\widetilde{v}(1-2^{-n-1})} < +\infty
\end{equation}
\end{enumerate}
\end{theorem}

We are now able to show the main results of this section.
\begin{theorem}
\label{thm:univconmap-Hinftyv}
Let $\Omega \subset \C$ be a hyperbolic domain with universal covering map $f_{\Omega}$, and let $v$ be a typical weight for which all composition operators are bounded on $\mathrm{H}^{\infty}_v$. Assume that $0 \in \Omega$. The following facts are equivalent:
\begin{enumerate}[\hspace{0.5cm}\rm(a)]
\item $\mathrm{Hol}(\D,\Omega) \subset \mathrm{H}^{\infty}_v$,
\item $f_{\Omega} \in \mathrm{H}^{\infty}_v$,
\item $\displaystyle\sup_{w \in \Omega}\abs{w}v\left(\dfrac{e^{d_{\Omega}(0,w)}-1}{e^{d_{\Omega}(0,w)}+1}\right) < +\infty$.
\end{enumerate}
\begin{proof}
Without loss of generality we may assume that $f_{\Omega}(0) = 0$.

(a) $\Longrightarrow$ (b). Since $f_{\Omega} \in \mathrm{Hol}(\D,\Omega)$, it is obvious that $f_{\Omega} \in \mathrm{H}^{\infty}_v$ whenever $\mathrm{Hol}(\D,\Omega) \subset \mathrm{H}^{\infty}_v$.

(b) $\Longrightarrow$  (a). Let $f \in \mathrm{Hol}(\D,\Omega)$. By property (I) in Subsection \ref{subsec:covering}, we can find a holomorphic map $g \colon \D \to \D$ such that $f = f_{\Omega} \circ g$. Since we are assuming that all composition operators are bounded on $\mathrm{H}^{\infty}_v$, if $f_{\Omega} \in \mathrm{H}^{\infty}_v$, it follows at once that $f \in \mathrm{H}^{\infty}_v$. Since $f$ is an arbitrary function of $\mathrm{Hol}(\D,\Omega)$, it follows that $\mathrm{Hol}(\D,\Omega) \subset \mathrm{H}^{\infty}_v$, and so the implication holds.

(b) $\Longleftrightarrow$  (c). For a moment, assume that $\Omega$ is multiply connected. Let $\Gamma$ be the group of automorphisms associated with $\Omega$ (see the property (II) in Subsection \ref{subsec:covering}) and let $\mathcal{D}$ be the Dirichlet polygon of $\Omega$ with center $0$ (see Subsection \ref{subsec:dirichlet}). Notice that
\begin{align}
\mathcal{D} & = \{z \in \D : d_{\D}(0,z) < d_{\D}(0,\gamma(z)) \text{ for all } \gamma \in \Gamma \setminus \{\mathrm{Id}_{\D}\}\} \notag \\
& = \{z \in \D : \abs{z} < \abs{\gamma(z)} \text{ for all } \gamma \in \Gamma \setminus \{\mathrm{Id}_{\D}\}\}, \label{eq:monotonic-D}
\end{align}
where we have used \eqref{eq:hypdist-D} and the fact that $d_{\D}(0,z) = d_{\D}(0,\abs{z})$ is increasing with respect to $\abs{z}$. We recall that, given $z,w \in \D$, it holds that $f_{\Omega}(z) = f_{\Omega}(w)$ if and only if there exists $\gamma \in \Gamma$ such that $w = \gamma(z)$. In that case, if $z \in \overline{\mathcal{D}} \cap \D$, then $\abs{z} \leq \abs{w}$ and so $v(z) \geq v(w)$, since $v$ is radial and non-increasing. Therefore,
\begin{align}
\label{eq:sup-dirichlet}
\sup_{z \in \D}\abs{f_{\Omega}(z)}v(z) & = \sup_{z \in \D}\abs{f_{\Omega}(z)}v(\abs{z}) = \sup_{z \in \overline{\mathcal{D}} \cap \D}\abs{f_{\Omega}(z)}v(\abs{z}) \\
& = \sup_{z \in \mathcal{D}}\abs{f_{\Omega}(z)}v(\abs{z}). \notag
\end{align}

But, for $z \in \mathcal{D}$, we can use \eqref{eq:domega-inf} and the fact that $f_{\Omega}(0) = 0$ to see that
$$\abs{z} = \dfrac{e^{d_{\D}(0,z)}-1}{e^{d_{\D}(0,z)}+1} = \dfrac{e^{d_{\Omega}(0,f_{\Omega}(z))}-1}{e^{d_{\Omega}(0,f_{\Omega}(z))}+1}.$$
Using this fact and \eqref{eq:sup-dirichlet}, we conclude that
\begin{equation}
\label{eq:sup-dirichlet-2}
\sup_{z \in \mathcal{D}}\abs{f_{\Omega}(z)}v(z) = \sup_{w \in f_{\Omega}(\mathcal{D})}\abs{w}v\left(\dfrac{e^{d_{\Omega}(0,w)}-1}{e^{d_{\Omega}(0,w)}+1}\right) = \sup_{w \in \Omega}\abs{w}v\left(\dfrac{e^{d_{\Omega}(0,w)}-1}{e^{d_{\Omega}(0,w)}+1}\right),
\end{equation}
where we have also used that $f_{\Omega}(\mathcal{D})$ is dense in $\Omega$ (see Property (b) in Subsection \ref{subsec:dirichlet}). Thus, $f_{\Omega} \in \mathrm{H}^{\infty}_v$ if and only if (c) holds.

In the case that $\Omega$ is simply connected, we see that \eqref{eq:sup-dirichlet-2} also holds with the convention that $\mathcal{D} = \D$ (see Remark \ref{remark:simply-connected}), and so the proof follows similarly.
\end{proof}
\end{theorem}

In the literature, the spaces
$$\mathrm{H}^0_v = \{f \in \mathrm{Hol}(\D,\C) : \lim_{\abs{z} \to 1} \abs{f(z)}v(z) = 0\}$$
where $v$ is a typical weight are also well-studied. For instance, as noticed in \cite[Theorem 2.3]{BDLT}, the typical weights $v$ for which all composition operators in $\mathrm{H}^0_v$ are bounded coincide with those for which all composition operators in $\mathrm{H}^{\infty}_v$ are bounded (see Theorem \ref{thm:continuity-Cg-Hinftyv}). In fact, we notice that the methods used in the proof of Theorem \ref{thm:univconmap-Hinftyv} can be adapted to obtain the following result.

\begin{theorem}
\label{thm:H0v-limsup0}
Let $\Omega \subset \C$ be a hyperbolic domain with universal covering map $f_{\Omega}$, and let $v$ be a typical weight satisfying for which all composition operators are bounded on $\mathrm{H}^{0}_v$. Assume that $0 \in \Omega$. The following facts are equivalent:
\begin{enumerate}[\hspace{0.5cm}\rm(a)]
\item $\mathrm{Hol}(\D,\Omega) \subset \mathrm{H}^0_v$,
\item $f_{\Omega} \in \mathrm{H}^0_v$,
\item $\displaystyle\limsup_{\substack{w \to \infty \\ w \in \Omega}}\abs{w}v\left(\dfrac{e^{d_{\Omega}(0,w)}-1}{e^{d_{\Omega}(0,w)}+1}\right) = 0$.
\end{enumerate}
\begin{proof}
Without loss of generality, we may assume that $f_{\Omega}(0) = 0$.

(a) $\Longleftrightarrow$ (b). As before, it follows using the same arguments as in the proof of Theorem \ref{thm:univconmap-Hinftyv}.

(b) $\Longleftrightarrow$ (c). First of all, we claim that (b) holds if and only if
\begin{equation}
\label{eq:member-H0v-Dirichlet}
\lim_{\substack{\abs{z} \to 1 \\ z \in \mathcal{D}}}\abs{f_{\Omega}(z)}v(z) = 0,
\end{equation}
where $\mathcal{D}$ is the Dirichlet polygon of $\Omega$ with center $0$. If $\Omega$ is simply connected, this is obvious since $\mathcal{D} = \D$. For a multiply connected domain, it is also clear that (b) implies \eqref{eq:member-H0v-Dirichlet}.

For the converse, we have to prove that \eqref{eq:member-H0v-Dirichlet} implies that
\begin{equation}
\label{eq:to-prove-2}
\lim_{n \to \infty}\abs{f_{\Omega}(z_n)}v(z_n) = 0
\end{equation}
for every sequence $\{z_n\} \subset \D$ with $\abs{z_n} \to 1$, as $n \to \infty$. By density, we can assume that $\{z_n : n \in \N\} \subset \bigcup_{\gamma \in \Gamma}\gamma(\mathcal{D})$, where $\Gamma$ is the group of automorphisms of $\D$ associated to $\Omega$. In order prove \eqref{eq:to-prove-2}, for every $n \in \N$ consider $\gamma_n \in \Gamma$ such that $w_n = \gamma_n(z_n) \in \mathcal{D}$. By the very definition of $\mathcal{D}$, it holds that $\abs{w_n} \leq \abs{z_n}$. Notice that
$$f_{\Omega}(w_n) = f_{\Omega}(z_n), \qquad v(w_n) \geq v(z_n),$$
where we have used the monotonicity of $v$. Now, let $\varepsilon > 0$ be arbitrary. By \eqref{eq:member-H0v-Dirichlet}, find $R \in (0,1)$ so that $\abs{f_{\Omega}(z)}v(z) < \varepsilon$ for all $z \in \mathcal{D}$ with $\abs{z} \geq R$. Let us define
$$A = \{n \in \N : \abs{w_n} < R\}, \qquad B = \N \setminus A.$$
If $n \in B$, then $\abs{f_{\Omega}(z_n)}v(z_n) \leq \abs{f_{\Omega}(w_n)}v(w_n) \leq \varepsilon$. On the other hand, there exists $C > 0$ so that $\abs{f_{\Omega}(z_n)} \leq C$ for all $n \in A$. But, since $\abs{z_n} \to 1$, as $n \to \infty$, it holds that $v(z_n) \leq \varepsilon/C$ for all $n \in A$ up to a finite number of terms. Summing up, we have shown that
$$\abs{f_{\Omega}(z_n)}v(z_n) \leq \varepsilon$$
for all $n \in \N$ up to a finite number of terms. Since $\varepsilon > 0$ is arbitrary, this proves the claim.

Let $z \in \mathcal{D}$ and let $w = f_{\Omega}(z)$. Recall that
$$\abs{f_{\Omega}(z)}v(z) = \abs{w}v\left(\dfrac{e^{d_{\Omega}(0,w)}-1}{e^{d_{\Omega}(0,w)}+1}\right).$$
We claim that
\begin{align*}
\limsup_{\substack{\abs{z} \to 1 \\ z \in \mathcal{D}}}\abs{f_{\Omega}(z)}v(z) & = \limsup_{\substack{w \to \infty \\ w \in f_{\Omega}(\mathcal{D})}}\abs{w}v\left(\dfrac{e^{d_{\Omega}(0,w)}-1}{e^{d_{\Omega}(0,w)}+1}\right) \\
& = \limsup_{\substack{w \to \infty \\ w \in \Omega}}\abs{w}v\left(\dfrac{e^{d_{\Omega}(0,w)}-1}{e^{d_{\Omega}(0,w)}+1}\right).    
\end{align*}
The second equality simply follows because $f_{\Omega}(\mathcal{D})$ is dense in $\Omega$. For the first equality, it is obvious that
$$\limsup_{\substack{\abs{z} \to 1 \\ z \in \mathcal{D}}}\abs{f_{\Omega}(z)}v(z) \geq \limsup_{\substack{w \to \infty \\ w \in f_{\Omega}(\mathcal{D})}}\abs{w}v\left(\dfrac{e^{d_{\Omega}(0,w)}-1}{e^{d_{\Omega}(0,w)}+1}\right).$$

For the converse inequality, let $\{z_n\} \subset \mathcal{D}$ be a sequence with $\abs{z_n} \to 1$, as $n \to \infty$, such that
$$L = \limsup_{\substack{\abs{z} \to 1 \\ z \in \mathcal{D}}}\abs{f_{\Omega}(z)}v(z)  = \lim_{n \to \infty}\abs{f_{\Omega}(z_n)}v(z_n).$$
If $\sup_{n \in \N}\abs{f_{\Omega}(z_n)} < +\infty$, then $L = 0$ since $v$ is typical, and then the claim holds.
Otherwise, $\sup_{n \in \N}\abs{f_{\Omega}(z_n)} = +\infty$. By taking a subsequence if needed, we may suppose that $\lim_{n \to \infty}f_{\Omega}(z_n) = \infty$. But then, by construction, we have that
$$L = \lim_{n \to \infty}\abs{f_{\Omega}(z_n)}v(z_n) \leq \limsup_{\substack{w \to \infty \\ w \in \Omega}}\abs{w}v\left(\dfrac{e^{d_{\Omega}(0,w)}-1}{e^{d_{\Omega}(0,w)}+1}\right),$$
and so the claim also holds.

It is clear that the proof follows from the above claims.
\end{proof}
\end{theorem}

It is worth pointing out that there are similar results dealing with the membership of conformal maps (i.e., in the setting of simply connected domains) to Hardy spaces \cite[Theorem 1.1]{K-Hardy} and Bergman spaces \cite[Theorem 1.1]{BKK} in terms of the hyperbolic distance. For the general setting of hyperbolic domains and the membership of their universal covering maps in Hardy spaces, we also refer to \cite[Lemma 3.2]{BCZ}. As far as we know, such a general result for Bergman spaces is not currently known.

\section{The growth number of a domain}
\label{sec:growth-number}
In this section we mainly deal with the weights $v_{\alpha}$, $\alpha > 0$, in order to study the growth number of a domain, and we prove Theorem \ref{thm:growth-number}. Recall that such weights are typical and, in fact, they also satisfy that $v_{\alpha} = \widetilde{v_{\alpha}}$ (see \cite[Examples 1.7]{BBT}). In particular, all of them satisfy \eqref{eq:continuity-wtv}. In other words, all weights $v_{\alpha}$, $\alpha > 0$, satisfy the hypotheses in Theorem \ref{thm:univconmap-Hinftyv}.

To this extent, we first notice that, if $v$ is a typical weight, then $\mathrm{Hol}(\D,\Omega) \subset \mathrm{H}^{\infty}_v$ whenever $\Omega \subset \C$ is a bounded domain. As a consequence, $\g(\Omega) = 0$ whenever $\Omega$ is bounded. Thus, we can restrict our interest to unbounded domains. Moreover, the cases in which $\Omega$ is not a hyperbolic domain are also covered by the following result:
\begin{lemma}
\label{lemma:growth-bounded-complement}
Let $\Omega \subset \C$ be a domain such that $\C \setminus \Omega$ is bounded. Then, $\g(\Omega) = +\infty$.
\begin{proof}
The proof is analogous to the ones given in the setting of the Hardy spaces \cite[Lemma 2.3.(4)]{KimSugawa} and the Bergman spaces \cite[Lemma 2.2]{Karafyllia}.

Let $\alpha > 0$. Since $\mathrm{H}^{\infty}_{v_{\alpha}}$ is a vector space, we may assume that $\C \setminus \Omega \subset \D$. In that case, it is easy to see that the map $z \mapsto \exp((1+z)/(1-z))$ belongs to $\mathrm{Hol}(\D,\Omega)$ but not to $\mathrm{H}^{\infty}_{v_{\alpha}}$, since
$$\lim_{x \to 1^-} (1-x^2)^{\alpha}\exp\left(\dfrac{1+x}{1-x}\right) = +\infty.$$
\end{proof}
\end{lemma}

In order to discuss the case of unbounded hyperbolic domains, we now re-state Theorem \ref{thm:univconmap-Hinftyv} in the case that $v = v_{\alpha}$.
\begin{corollary}
\label{cor:va-hypdist}
Let $\Omega \subset \C$ be a hyperbolic domain, and let $\alpha > 0$. Assume that $0 \in \Omega$. The following facts are equivalent:
\begin{enumerate}[\hspace{0.5cm} \rm (a)]
\item $\mathrm{Hol}(\D,\Omega) \subset \mathrm{H}^{\infty}_{v_{\alpha}}$,
\item $\displaystyle\sup_{w \in \Omega }\abs{w} e^{-\alpha d_{\Omega}(0,w)}< +\infty$. That is, there exists $C > 0$ such that
$$d_{\Omega}(0,w) \geq \dfrac{1}{\alpha}\log\abs{w} - C, \qquad  w \in \Omega.$$
\end{enumerate}
\begin{proof}
It follows at once from Theorem \ref{thm:univconmap-Hinftyv} simply by noticing that
$$e^{-\alpha d_{\Omega}(0,w)} \leq v_{\alpha}\left(\dfrac{e^{d_{\Omega}(0,w)}-1}{e^{d_{\Omega}(0,w)}+1}\right) = \left(\dfrac{2}{e^{d_{\Omega}(0,w)}+1}\right)^{\alpha} \leq 2^{\alpha}e^{-\alpha d_{\Omega}(0,w)},$$
since
$$e^{d_{\Omega}(0,w)} \leq e^{d_{\Omega}(0,w)}+1 \leq 2e^{d_{\Omega}(0,w)}, \qquad w \in \Omega.$$
\end{proof}
\end{corollary}


We now prove Theorem \ref{thm:growth-number}.
\begin{proof}[Proof of Theorem \ref{thm:growth-number}]
The equality between the two upper limits is obvious. To relate them with the growth number of $\Omega$, we proceed as follows:

Let
$$L = \limsup_{\substack{w \to \infty \\ w \in \Omega}}\dfrac{\log(\abs{w})}{d_{\Omega}(0,w)},$$
and let
$$I = \inf(\{\alpha > 0 : \sup_{w \in \Omega}\exp(-\alpha d_{\Omega}(0,w))\abs{w} < +\infty\}).$$
By Corollary \ref{cor:va-hypdist}, it is clear that
$\g(\Omega) = I$. 
So, it suffices to prove that $L=I$.
To do so, first assume that $\alpha > L$. This means that there exists $R > 0$ so that, for all $w \in \Omega$ with $\abs{w} > R$ we have that
$$\dfrac{\log(\abs{w})}{d_{\Omega}(0,w)} < \alpha.$$
That is, if $w \in \Omega$ with $\abs{w} > R$, then
$$e^{-\alpha d_{\Omega}(0,w)}\abs{w} < 1.$$
In particular,
$$\sup_{w \in \Omega}e^{-\alpha d_{\Omega}(0,w)}\abs{w} < +\infty,$$
which means that $\alpha \geq I$. Since $\alpha > L$ is arbitrary, we see that $L \geq I$. 

Finally, assume that $\alpha < L$. This means that there exist $\varepsilon > 0$ and a sequence $w_n \in \Omega$, $n \in \N$, such that
$$\lim_{n \to \infty}w_n = \infty, \qquad \dfrac{\log(\abs{w_n})}{d_{\Omega}(0,w_n)} > \alpha + \varepsilon.$$
In other words,
$$e^{-\alpha d_{\Omega}(0,w_n)}\abs{w_n} > e^{\varepsilon d_{\Omega}(0,w_n)} \to +\infty, \qquad \text{as } n \to \infty.$$
In particular,
$$\sup_{w \in \Omega}e^{-\alpha d_{\Omega}(0,w)}\abs{w} = +\infty,$$
which means that $\alpha \leq I$. Since $\alpha < L$ is arbitrary, we see that $L \leq I$. 
\end{proof}

Following the spirit of Theorem \ref{thm:H0v-limsup0}, for $\alpha > 0$ we notice that
$$\mathrm{H}^{\infty}_{v_{\beta}} \subset \mathrm{H}^0_{v_{\alpha}} \subset \mathrm{H}^{\infty}_{v_{\alpha}}, \qquad \beta < \alpha.$$
Therefore, for every domain $\Omega \subset \C$ we have that
$$\g(\Omega) = \inf(\{\alpha > 0 : \mathrm{Hol}(\D,\Omega) \subset \mathrm{H}^0_{v_{\alpha}}\}),$$
where $\inf(\emptyset) = +\infty$.

\section{Domains with a prefixed growth number}
\label{sec:examples}
This section is devoted to establish Theorem \ref{thm:examples}.(a). In the case of Hardy spaces, the analogue result was first proven in \cite[Theorem 5.2]{CCZKRP}. In view of Theorem \ref{thm:growth-number}, the latter result will be established by constructing domains with a controlled hyperbolic metric. 

We first notice that examples in the range $0 \leq \g(\Omega) \leq 2$ are easy to construct via simply connected domains. In particular, we establish a universal upper bound for their growth number.
\begin{lemma}
\label{lemma:growth-examples-simpconn}
Let $\Omega \subsetneq \C$ be a simply connected domain. Then, $\g(\Omega) \leq 2$.
Moreover, for every $\alpha \in [0,2]$ there exists a simply connected domain $\Omega \subsetneq \C$ such that $\g(\Omega) =~\alpha$.
\begin{proof}
Let $f_{\Omega}$ be a conformal map of a simply connected domain $\Omega \subsetneq \C$. By suitably translating $\Omega$ (notice that this leaves the growth number invariant), we may assume that $f_{\Omega}(0) = 0$. We also may assume that $f'_{\Omega}(0) > 0$. In that case, the classical Koebe Distortion Theorem \cite[Theorem 1.6]{PommerenkeUnivFun} asserts that
$$\abs{\dfrac{f_{\Omega}(z)}{f'_{\Omega}(0)}} \leq \dfrac{\abs{z}}{(1-\abs{z})^2}, \qquad z \in \D.$$
It easily follows that $f_{\Omega} \in \mathrm{H}^{\infty}_{v_{\alpha}}$ for all $\alpha \geq 2$, and so the first part of the result follows from Theorem \ref{thm:univconmap-Hinftyv}.

For the second part of the lemma, let $\gamma \in (0,2]$ and define $f_{\gamma} \colon \D \to \C$ given by
$$f_{\gamma}(z) = \left(\dfrac{1+z}{1-z}\right)^\gamma, \qquad z \in \D.$$
Notice that $f_{\gamma}$ is univalent, and so it is the conformal map of the simply connected domain $\Omega_{\gamma} = f_{\gamma}(\D)$, which is an angular region. It is straightforward to see that $f_{\gamma} \in \mathrm{H}^{\infty}_{v_{\alpha}}$ if and only if $\alpha \geq \gamma$. Therefore, using Theorem \ref{thm:univconmap-Hinftyv}, we conclude that $\g(\Omega_{\gamma}) = \gamma \in (0,2]$.

The remaining case $\g(\Omega) = 0$ is achieved, for example, by bounded domains.
\end{proof}
\end{lemma}


The examples in the range $2 < \g(\Omega) < +\infty$ are more involved. To introduce them, let $\gamma > 1$ and define
\begin{equation}
\label{eq:domains}
\Omega_{\gamma} = \C \setminus (\{0\} \cup \{-\gamma^n : n \in \Z\}).
\end{equation}
Notice that $\gamma\Omega_{\gamma} = \Omega_{\gamma}$. We start by proving the following lemma:
\begin{lemma}
\label{lemma:g-gamma}
With the notations above,
$$\g(\Omega_{\gamma}) = \dfrac{\log(\gamma)}{d_{\Omega_{\gamma}}(1,\gamma)}, \qquad \gamma > 1.$$
\begin{proof}
Let $\widetilde{\Omega}_{\gamma} = \Omega_{\gamma} - 1$. Then  $\g(\widetilde{\Omega}_{\gamma}) = \g(\Omega_{\gamma})$. By Theorem \ref{thm:growth-number}, we know that 
$$\g(\Omega_{\gamma}) =\g(\widetilde{\Omega}_{\gamma}) = \limsup_{R \to + \infty}\dfrac{\log(R)}{d_{\widetilde{\Omega}_{\gamma}}(0,\widetilde{F}_R)},$$
where $\widetilde{F}_R = \{z \in \widetilde{\Omega}_{\gamma} : \abs{z} = R\}$, $R > 0$.

Since $\widetilde{\Omega}_{\gamma}$ is invariant by circular symmetrization, we can use \cite[Theorem 2]{W} to deduce that $d_{\widetilde{\Omega}_{\gamma}}(0,\widetilde{F}_R) = d_{\widetilde{\Omega}_{\gamma}}(0,R) = d_{\Omega_{\gamma}}(1,R+1)$ for $R > 0$. Analogously, we can also see that the segment $[a,b]$ is a hyperbolic geodesic of $\Omega_{\gamma}$ whenever $a,b > 1$. In particular, we deduce that
\begin{equation}
\label{eq:correction1}
\g(\Omega_{\gamma}) = \limsup_{R \to + \infty}\dfrac{\log(R)}{d_{\Omega_{\gamma}}(1,R+1)}= \limsup_{R \to + \infty}\dfrac{\log(R)}{d_{\Omega_{\gamma}}(1,R)}.
\end{equation}
Moreover, for $R > \gamma$, 
$$d_{\Omega_{\gamma}}(1,R) = \sum_{n = 1}^N d_{\Omega_{\gamma}}(\gamma^{n-1},\gamma^n) + d_{\Omega_{\gamma}}(\gamma^N,R),$$
being $N = \left\lfloor \log(R) / \log(\gamma) \right\rfloor$, where $ \left\lfloor x \right\rfloor$ denotes the greatest integer smaller than or equal to $x \in \R$. Since $\gamma\Omega _{\gamma} = \Omega_{\gamma}$, we have that $d_{\Omega_{\gamma}}(\gamma^{n-1},\gamma^n) = d_{\Omega_{\gamma}}(1,\gamma)$. Summing up,
\begin{equation}
\label{eq:correction2}
d_{\Omega_{\gamma}}(1,R) = N d_{\Omega_{\gamma}}(1,\gamma) + d_{\Omega_{\gamma}}(\gamma^N,R),
\end{equation}
where $0 \leq d_{\Omega_{\gamma}}(\gamma^N,R) = d_{\Omega_{\gamma}}(1,R/\gamma^N) < d_{\Omega_{\gamma}}(1,\gamma)$. Then, the result follows at once from \eqref{eq:correction1} and \eqref{eq:correction2}.
\end{proof}
\end{lemma}

We now present the main argument towards the proof of Theorem \ref{thm:examples}.
\begin{lemma}
\label{lemma:g-Omega-gamma}
With the notations above, we have that:
\begin{enumerate}[\hspace{0.5cm} \rm (a)]
\item The map $\gamma \mapsto \g(\Omega_{\gamma})$, $\gamma > 1$, is continuous.
\item $\lim_{\gamma \to +\infty}\g(\Omega_{\gamma}) = +\infty$.
\item $\lim_{\gamma \to 1^+}\g(\Omega_{\gamma}) = 2$.
\end{enumerate}
\begin{proof}

(a) Using Lemma \ref{lemma:g-gamma}, it is enough to prove that the map $\gamma \mapsto d_{\Omega_{\gamma}}(1,\gamma)$, $\gamma > 1$, is continuous.
To do so, fix $\gamma > 1$ and let $\gamma_n \in (1,+\infty)$ be a sequence with $\gamma_n \to \gamma$, as $n \to \infty$. It is clear that the kernel of $\{\Omega_{\gamma_n}\}$ with respect to $1$ is $\Omega_{\gamma}$. In particular, it is easy to check that $\{\Omega_{\gamma_n}\}$ converges to its kernel with respect to $1$. Then, the result follows from Corollary \ref{cor:distance-convergence} at once.

(b) Let $\gamma > 1$. If $x \in [1,\gamma]$, then $\mathrm{dist}(x,\partial\Omega_\gamma ) = x$ and 
$$\beta_{\Omega_{\gamma}}(x) = \inf\left\lbrace\abs{\log\left(\dfrac{x}{\gamma^n}\right)} : n \in \Z \right\rbrace = \min \left\lbrace \log(x), \, \log\left(\dfrac{\gamma}{x}\right)\right\rbrace,$$
where $\beta_{\Omega_{\gamma}}$ is the function defined in Subsection \ref{sec:BeardonPommerenke}. Therefore, applying the estimate in Theorem \ref{thm:BeardonPommerenke}, we have that
\begin{align*}
d_{\Omega_{\gamma}}(1,\gamma) & = \int_1^{\gamma} \lambda_{\Omega_{\gamma}}(x)dx \leq C\int_1^{\gamma}\dfrac{dx}{\mathrm{dist}(x,\partial\Omega_{\gamma})(\kappa+\beta_{\Omega_{\gamma}}(x))} \\
& = C\left(\int_1^{\sqrt{\gamma}}\dfrac{dx}{x(\kappa+\log(x))} + \int_{\sqrt{\gamma}}^{\gamma}\dfrac{dx}{x(\kappa+\log(\gamma)-\log(x))}\right) \\
& = 2C\left(\log\left(\kappa+\dfrac{1}{2}\log(\gamma)\right)-\log(\kappa)\right),
\end{align*}
where $C = \kappa+\pi/4 > 0$. In particular,
$$\lim_{\gamma \to +\infty} \dfrac{\log(\gamma)}{d_{\Omega_{\gamma}}(1,\gamma)} = +\infty,$$
and so the result follows from Lemma \ref{lemma:g-gamma}.

(c) Let $\gamma_n > 1$ forming a sequence such that $\gamma_n \to 1$, as $n \to +\infty$. Notice that for every $x \in (-\infty,0)$, there exists a sequence $\{x_n\}$ such that $x_n \not\in\Omega_{\gamma_n}$ for all $n \in \N$, but $x_n \to x$, as $n \to \infty$. Therefore, the kernel of $\{\Omega_{\gamma_n}\}$ with respect to $1$ is $\Omega := \C \setminus (-\infty,0]$, and $\{\Omega_{\gamma_n}\}$ converges to $\Omega$.

Since $\gamma_n \to 1$ as $n \to \infty$, for every $n \in \N$ large enough there exists $m_n \in \N$ with $\gamma_n^{m_n} \leq 2 < \gamma_n^{m_n+1}$. Notice that
$$2-\gamma_n^{m_n} \leq \gamma_n^{m_n+1}-\gamma_n^{m_n} = \gamma_n^{m_n}(\gamma_n-1) \leq 2(\gamma_n-1),$$
and so $\gamma_n^{m_n} \to 2$ as $n \to \infty$. Moreover, using again that $\gamma_n\Omega _{\gamma_n} = \Omega_{\gamma_n}$, we have
$$
d_{\Omega_{\gamma_n}}(1,\gamma_n^{m_n}) = \sum_{k = 1}^{m_n} d_{\Omega_{\gamma_n}}(\gamma_n^{k-1},\gamma_n^k)= \sum_{k = 1}^{m_n} d_{\Omega_{\gamma_n}}(1,\gamma_n)= m_nd_{\Omega_{\gamma_n}}(1,\gamma_n).
$$
Thus,
\begin{align}
\label{eq:referee}
\lim_{n \to +\infty}\dfrac{\log(\gamma_n)}{d_{\Omega_{\gamma_n}}(1,\gamma_n)} = \lim_{n \to +\infty}\dfrac{m_n\log(\gamma_n)}{m_nd_{\Omega_{\gamma_n}}(1,\gamma_n)} = \lim_{n \to +\infty}\dfrac{\log(\gamma_n^{m_n})}{d_{\Omega_{\gamma_n}}(1,\gamma_n^{m_n})} = \dfrac{\log(2)}{d_{\Omega}(1,2)},
\end{align}
where we have also used Corollary \ref{cor:distance-convergence}. But
$$d_{\Omega}(1,x) = d_{\D}\left(0,\dfrac{\sqrt{x}-1}{\sqrt{x}+1}\right) = \dfrac{1}{2}\log(x), \qquad x > 1,$$
where we have used \eqref{eq:hypdist-D} and \eqref{eq:domega-inf}. So, $d_\Omega (1,2) = \log(2)/2$. This in combination with \eqref{eq:referee} and Lemma \ref{lemma:g-gamma} gives $\lim_{\gamma \to 1^+} \g(\Omega_\gamma)=2$.
\end{proof}
\end{lemma}

Now, Theorem \ref{thm:examples}.(a) follows by Lemma \ref{lemma:growth-examples-simpconn} and Lemma \ref{lemma:g-Omega-gamma}.
\begin{proof}[Proof of Theorem \ref{thm:examples}.(a)]
For $\alpha = +\infty$, simply choose $\Omega = \C$. For $\alpha \in [0,2]$, see Lemma \ref{lemma:growth-examples-simpconn}. For $\alpha \in (2,+\infty)$, see Lemma \ref{lemma:g-Omega-gamma}.
\end{proof}

\section{Relation between the Bergman number and the growth number of a domain}
\label{sec:bergman}
This section is devoted to prove Theorem \ref{thm:bergman-growth} using Theorem \ref{thm:growth-number} and the inclusion among Bergman spaces and growth spaces. We also finish the proof of Theorem \ref{thm:examples}. Before doing so, we recall that a different (but related) definition of the Bergman number of a domain $\Omega \subset \C$ was proposed in \cite{BCZ}. Namely, for $\alpha > -1$, we define
$$\b_{\alpha}(\Omega) = \sup(\{p > 0: \mathrm{Hol}(\D,\Omega) \subset A^p_{\alpha}\}),$$
where we use the convention that $\sup(\emptyset) = 0$.
For every domain $\Omega \subset \C$ it is possible to argue \cite[Lemma 2.14]{BCZ} that
\begin{equation}
\label{eq:BCZ-bab}
\dfrac{\b_{\alpha}(\Omega)}{\alpha+2} \leq \b(\Omega), \qquad \alpha > -1.
\end{equation}
We will refer to both $\b(\Omega)$ and $\b_{\alpha}(\Omega)$ as the Bergman number of $\Omega$, trying to avoid confusion by providing the appropriate context.

We begin by reviewing the following inclusions.
\begin{lemma}
\label{lemma:growth-bergman-inclusions}
The following inclusions hold:
\begin{enumerate}[\hspace{0.5cm}\rm(a)]
\item Let $p >0$ and $\alpha > -1$. Then, $A^p_{\alpha} \subset \mathrm{H}^{\infty}_{v_{\beta}}$, where $\beta = (\alpha+2)/p$.
\item Let $\beta > 0$. Then, $\mathrm{H}^{\infty}_{v_{\beta}} \subset A^p_{\alpha}$ for all $p > 0$ and $\alpha > -1$ such that $p/(\alpha+1) < 1/\beta$.
\end{enumerate}
\begin{proof}
(a) It follows from \cite[Proposition 7.1.1 and p. 148]{Vukotic}.

(b) If $f \in \mathrm{H}^{\infty}_{v_{\beta}}$, then
$$\int_{\D}(1-\abs{z}^2)^{\alpha}\abs{f(z)}^pdA(z) \leq \int_{\D}\dfrac{MdA(z)}{(1-\abs{z}^2)^{\beta p - \alpha}} = 2\pi M \int_0^1\dfrac{rdr}{(1-r^2)^{\beta p -\alpha}},$$
for some $M > 0$. Then, $f \in A^p_{\alpha}$ whenever $\beta p-\alpha < 1$.
\end{proof}
\end{lemma}

The latter inclusions yield the following consequence.
\begin{lemma}
\label{lemma:estimates}
Let $\Omega \subset \C$ be a domain and $\alpha > -1$. The following facts hold:
\begin{enumerate}[\hspace{0.5cm}\rm(a)]
\item $\g(\Omega) = +\infty$ if and only if $\b_{\alpha}(\Omega) = 0$, and if and only if $\b(\Omega) = 0$.
\item $\g(\Omega) = 0$ if and only if $\b_{\alpha}(\Omega) = +\infty$, and if and only if $\b(\Omega) = +\infty$.
\item If $0 < \g(\Omega) < +\infty$, then $\alpha + 1 \leq \g(\Omega)\b_{\alpha}(\Omega) \leq \alpha + 2$.
\item If $0 < \g(\Omega) < +\infty$, then $\g(\Omega)\b(\Omega) = 1$.
\end{enumerate}
\begin{proof}
(a) Notice that $\g(\Omega) = +\infty$ if and only if $\mathrm{Hol}(\D,\Omega) \not\subset \mathrm{H}^{\infty}_{v_{\beta}}$ for any $\beta > 0$. By Lemma \ref{lemma:growth-bergman-inclusions}, this is also equivalent to $\mathrm{Hol}(\D,\Omega) \not\subset A^p_{\alpha}$ for any $p > 0$, which holds if and only if $\b_{\alpha}(\Omega) = 0$.
Arguing similarly, one also proves that the latter is equivalent to $\b(\Omega) = 0$.

(b) The proof follows by reasoning as above.

(c) Since $0 < \g(\Omega) < +\infty$, by (a) and (b) we have that $0 < \b_{\alpha}(\Omega) < +\infty$. In that case, choose $0 < p < \b_{\alpha}(\Omega)$. If $f \in \mathrm{Hol}(\D,\Omega)$, we have that $f \in A^p_{\alpha}$, which yields that $f \in \mathrm{H}^{\infty}_{v_{(\alpha+2)/p}}$ by Lemma \ref{lemma:growth-bergman-inclusions}. Since $f$ is arbitrary, we deduce that $\g(\Omega) \leq (\alpha+2)/p$. Letting $p \to \b_{\alpha}(\Omega)$ we deduce that
$$\g(\Omega) \leq \dfrac{\alpha+2}{\b_{\alpha}(\Omega)}.$$

Analogously, let $\beta > \g(\Omega)$. By definition, if $f \in \mathrm{Hol}(\D,\Omega)$, we have that $f \in \mathrm{H}^{\infty}_{v_{\beta}}$. By Lemma \ref{lemma:growth-bergman-inclusions}, $f \in A^{p_{\varepsilon}}_{\alpha}$ for all $\varepsilon > 0$, where $p_{\varepsilon} = (\alpha+1)/(\beta+\varepsilon)$. Since $f$ is arbitrary, we have that $\b_{\alpha}(\Omega) \geq p_{\varepsilon}$. Letting $\varepsilon \to 0$ and $\beta \to \g(\Omega)$ we deduce that
$$\b_{\alpha}(\Omega) \geq \dfrac{\alpha+1}{\g(\Omega)}.$$

(d) Since $0 < \g(\Omega) < +\infty$, by (a) and (b) we have that $0 < \b(\Omega) < +\infty$. It follows from (c) and \eqref{eq:BCZ-bab} that, for every $\alpha > -1$, we have that
$$\g(\Omega)\b(\Omega) \geq \g(\Omega)\dfrac{\b_{\alpha}(\Omega)}{\alpha+2} \geq \dfrac{\alpha+1}{\alpha+2}.$$
Letting $\alpha \to +\infty$ we deduce that $\g(\Omega)\b(\Omega) \geq 1$.

For the reverse inequality, let $0 < q < \b(\Omega)$. By definition, for every $f \in \mathrm{Hol}(\D,\Omega)$ there exist $p > 0$ and $\alpha > -1$ with $p/(\alpha+2) = q$ such that $f \in A^p_{\alpha}$. By Lemma \ref{lemma:growth-bergman-inclusions}, $f \in \mathrm{H}^{\infty}_{v_{(\alpha+2)/p}} = \mathrm{H}^{\infty}_{v_{1/q}}$. Since $f$ is arbitrary, we conclude that $\g(\Omega) \leq 1/q$. Letting $q \to \b(\Omega)$, we deduce that $\g(\Omega) \leq 1/\b(\Omega)$, that is,  $\g(\Omega)\b(\Omega) \leq 1$.
\end{proof}
\end{lemma}

\begin{proof}[Proof of Theorem \ref{thm:examples}.(b)]
The above lemma shows that Theorem \ref{thm:examples}.(b) is a direct consequence of Theorem \ref{thm:examples}.(a).
\end{proof}
We finish this section noticing that Theorem \ref{thm:bergman-growth} is a consequence of Theorem \ref{thm:growth-number} and the previous lemma. We also highlight that  Theorem \ref{thm:growth-number} and Lemma \ref{lemma:estimates} imply that
$$\dfrac{\alpha+1}{\alpha+2}\b(\Omega) \leq \dfrac{\b_{\alpha}(\Omega)}{\alpha+2} \leq \b(\Omega),$$
for every domain $\Omega \subset \C$ and every $\alpha > -1$. In particular,
$$\lim_{\alpha \to +\infty}\dfrac{\b_{\alpha}(\Omega)}{\alpha+2} = \b(\Omega).$$

\section{Decoupling the Hardy and the Bergman number of a domain}
\label{sec:questions}
As commented in the introduction, Karafyllia and Karamanlis \cite{KK} proved that the Bergman number and the Hardy number of simply connected domains coincide. To elaborate on this result, let us recall that, for $p > 0$, the Hardy space $\mathrm{H}^p$ is the collection of all holomorphic maps $f \colon \D \to \C$ such that
$$\sup_{0 < r < 1}\int_{0}^{2\pi}\abs{f(re^{it})}^pdt < +\infty.$$
Using these spaces, for a domain $\Omega \subset \C$, we define the Hardy number of $\Omega$ as
$$\h(\Omega) = \sup(\{p > 0 : \mathrm{Hol}(\D,\Omega) \subset \mathrm{H}^p\}),$$
where we use the convention that $\sup(\emptyset) = 0$.

As we said before, one of the motivations for the results in Section \ref{sec:examples} is that only recently it was proved that for every $p \in [0,+\infty]$ there exists a domain $\Omega \subset \C$ such that $\h(\Omega) = p$ (see \cite[Theorem 5.2]{CCZKRP}). Possibly, the reason behind this is that the literature about this topic mainly concerns the case of simply connected domains. However, as a consequence of \cite[Theorem 3.16]{Duren-Hp}, it is known that $\h(\Omega) \geq 1/2$ whenever $\Omega \subsetneq \C$ is a simply connected domain (see also \cite[Lemma 2.3.(5)]{KimSugawa}). 

A similar question can be posed in the setting of the Bergman number of general domains. To introduce it, we reformulate a remarkable theorem of Karafyllia and Karamanlis \cite[Theorem 1.3.(3)]{KK}. Namely, consider a simply connected domain $\Omega \subsetneq \C$ with $0 \in \Omega$. By \cite[Theorem 1.1]{K-HypDist}, \cite[Eq. (1.5)]{KK}, \cite[Lemmas 2.12, 2.13, and 2.14]{BCZ}, \cite[Theorem 1.3.(3)]{KK} (or Theorem \ref{thm:bergman-growth}), and the conformal invariance of the hyperbolic distance (see Remark \ref{remark:simply-connected} and \eqref{eq:domega-inf}) we have that
$$\h(\Omega) = \dfrac{\b_{\alpha}(\Omega)}{\alpha+2} = \b(\Omega) = \liminf_{R \to +\infty}\dfrac{d_{\Omega}(0,F_R)}{\log(R)}, \qquad \alpha > -1,$$
where $F_R = \{w \in \Omega : \abs{w} = R\}$, $R > 0$.

All in all, it is clear that $\b(\Omega) = \b_{\alpha}(\Omega)/(\alpha+2) \geq 1/2$ whenever $\Omega \subsetneq \C$ is a simply connected domain and $\alpha > -1$. However, these lower bounds are not true in general, as it follows from Theorem \ref{thm:bergman-growth}, Theorem \ref{thm:examples}, and Lemma \ref{lemma:estimates}.(c). Indeed, the examples considered in \eqref{eq:domains} are of zero Hardy number. This can be seen as a consequence of \cite[Theorem 2.4.(i)]{KimSugawa}, for instance, since their complements are polar sets. Therefore, we can state the following result.
\begin{corollary}
\label{cor:examples}
For every $p \in (0,1/2)$ there exists a domain $\Omega \subset \C$ with the following properties:
\begin{enumerate}[\hspace{0.5cm}\rm (a)]
\item $\h(\Omega) = 0$.
\item $\b(\Omega) = p$.
\item $(\alpha+1)p \leq \b_{\alpha}(\Omega) \leq (\alpha+2)p$, for all $\alpha > -1$. 
\end{enumerate}
\end{corollary}
We notice that this corollary gives a positive solution to \cite[Problem 7.4]{BCZ}, which asked about the existence of domains $\Omega \subset \C$ for which $0 \leq \h(\Omega) < \b(\Omega) < +\infty$.

Also, it follows from Theorem \ref{thm:bergman-growth} and Corollary \ref{cor:examples} that \cite[Theorem 1.1]{K-HypDist} does not hold for general domains.

\section{Bloch-type spaces and the Bloch number of a domain}
\label{sec:bloch}
In this section we analyze a new class of spaces. Namely, the Bloch-type spaces $\mathcal{B}_v$, where $v$ is a weight. We say that $f \in \mathrm{Hol}(\D,\C)$ belongs to $\mathcal{B}_v$ if
$$\sup_{z \in \D}\abs{f'(z)}v(z) < +\infty.$$
That is, $f \in \mathcal{B}_v$ if and only if $f' \in \mathrm{H}^{\infty}_v$. Notice that $\mathcal{B} = \mathcal{B}_{v_1}$ is the classical Bloch space.

The following classical theorem characterizes the domains associated with the Bloch space.
\begin{theorem}
\label{thm:classical-Bloch}
{\normalfont (\cite[Section 29, Theorem 6]{SeidelWalsh}; see also \cite[Theorem 2.6]{Cima})}
Let $\Omega \subset \C$ be a domain. The following facts are equivalent:
\begin{enumerate}[\hspace{0.5cm}\rm(a)]
\item $\mathrm{Hol}(\D,\Omega) \subset \mathcal{B}$,
\item $\sup_{w \in \Omega} R(w) < +\infty$,
\end{enumerate}
where $R(w) = \sup(\{r > 0 : D(w,r) \subset \Omega\})$ and $D(w,r) = \{z \in \C : \abs{w-z} < r\}$.
\end{theorem}

Our purpose in this subsection is to complement the previous result in the setting of Bloch-type spaces for some weights $v$. The motivation comes from the following result of Lusky \cite{Lusky}, which was later refined by Abanin and Tien \cite{AT}. In order to present it, we recall that a weight $v$ is said to be essential if there exists $C > 0$ such that $\widetilde{v}(z) \leq Cv(z)$ for all $z \in \D$. 
In fact, a typical weight $v$ is essential if and only if $r \mapsto v(r)$ is equivalent to a log-convex function \cite[Theorem 2]{Bonet-Survey}.

\begin{theorem}
\label{thm:Lusky}
{\normalfont \cite[Theorems 2, 27, and 30]{Bonet-Survey}}
Let $v$ be a typical and essential weight. Suppose that there exists $k \in \N$ such that
\begin{equation}
\label{eq:Lusky}
\limsup_{n \to \infty}\dfrac{v(1-2^{-n-k})}{v(1-2^{-n})} < 1, \qquad \sup_{n \in \N}\dfrac{v(1-2^{-n})}{v(1-2^{-n-1})} < +\infty.
\end{equation}
Then, $\mathrm{H}^{\infty}_v = \mathcal{B}_{\hat{v}}$, where $\hat{v}(z) = (1-\abs{z})v(z)$, $z \in \D$.
\end{theorem}

Notice that essential weights satisfying \eqref{eq:Lusky} also satisfy \eqref{eq:continuity-wtv}. Then, using Theorems \ref{thm:univconmap-Hinftyv} and \ref{thm:Lusky}, we obtain the following consequence, which complements Theorem \ref{thm:classical-Bloch}.
\begin{proposition}
Let $v$ be a typical and essential weight satisfying \eqref{eq:Lusky}, and let $\hat{v}(z) = (1-\abs{z})v(z)$, $z \in \D$. Let $\Omega \subset \C$ be a hyperbolic domain with universal covering map $f_{\Omega}$. Assume that $0 \in \Omega$. The following facts are equivalent:
\begin{enumerate}[\hspace{0.5cm}\rm(a)]
\item $\mathrm{Hol}(\D,\Omega) \subset \mathcal{B}_{\hat{v}}$,
\item $f_{\Omega} \in \mathcal{B}_{\hat{v}}$,
\item $\displaystyle\sup_{w \in \Omega}\abs{w}v\left(\dfrac{e^{d_{\Omega}(0,w)}-1}{e^{d_{\Omega}(0,w)}+1}\right) < +\infty$.
\end{enumerate}
\end{proposition}

It is also interesting to decide whether $\mathrm{Hol}(\D,\Omega) \subset \mathcal{B}_v$ for those typical weights $v$ satisfying
$$\limsup_{r \to 1^-}\dfrac{v(r)}{1-r} = +\infty.$$
This is partly covered by the following result.
\begin{lemma}
\label{lemma:bloch-no-domains}
Let $v$ be a typical weight. Assume that
$$\int_0^1\dfrac{dr}{v(r)} < + \infty.$$
Then, every function in $\mathcal{B}_v$ has a continuous extension to the closed unit disk $\overline{\D}$. In particular, $\mathrm{Hol}(\D,\Omega) \not\subset \mathcal{B}_v$ for every domain $\Omega \subset \C$.
\begin{proof}
Let $f \in \mathcal{B}_v$. To show that $f$ has a continuous extension to the closed unit disk it is enough to show that for all $\varepsilon > 0$ there exists $r_0 \in (0,1)$ such that
$$\abs{f(r_1 e^{i\theta})- f(r_2 e^{i\theta})} < \varepsilon$$
for all $\theta \in \R$ and for  all $r_0<r_1<r_2<1$. But this follows immediately from the integrability of the continuous function $1/v$ because
$$\abs{f(r_1 e^{i\theta})- f(r_2 e^{i\theta})} = \abs{\int_{r_1}^{r_2}f'(se^{i\theta})e^{i\theta}ds} \leq \sup_{z \in \D}\abs{f'(z)}v(z) \int_{r_1}^{r_2}\dfrac{1}{v(s)}ds.$$
This means that $f$ has a continuous extension to the closed unit disk. In particular, $\mathcal{B}_v \subsetneq \mathrm{H}^{\infty}$.

Now, let $f \in \mathrm{H}^{\infty} \setminus \mathcal{B}_v$ and let $\Omega \subset \C$ be a domain. Since $\Omega$ is an open set, there exist $\alpha,\beta \in \C$ be such that $g = \alpha f+\beta \in \mathrm{Hol}(\D,\Omega)$. Notice that $g \in \mathrm{H}^{\infty} \setminus \mathcal{B}_v$. As we wanted to conclude, this means that $\mathrm{Hol}(\D,\Omega) \not\subset \mathcal{B}_v$.
\end{proof}
\end{lemma}

As before, we are particularly interested in the case of the typical weights $v_{\alpha}$, $\alpha > 0$. Since $\mathcal{B}_{v_{\alpha}}$ is an increasing chain of spaces, mirroring the definition of the growth number of a domain, let us define the Bloch number of a domain $\Omega \subset \C$ as
$$\bloch(\Omega) = \inf(\{\alpha > 0 : \mathrm{Hol}(\D,\Omega) \subset \mathcal{B}_{v_{\alpha}}\}),$$
where we use the convention that $\inf(\emptyset) = +\infty$.

Recall that the weights $v_{\alpha}$, $\alpha > 0$, are essential. In fact, they also satisfy \eqref{eq:Lusky}. In particular, by Theorem \ref{thm:Lusky}, this means that $\mathcal{B}_{v_{\alpha}} = \mathrm{H}^{\infty}_{v_{\alpha-1}}$ for all $\alpha > 1$ (see \cite[Proposition 7]{Zhu} for a more elementary proof). For $0 < \alpha < 1$, our study is covered by Lemma \ref{lemma:bloch-no-domains}. In fact, as originally noticed by Hardy and Littlewood, the spaces $\mathcal{B}_{v_{\alpha}}$ can be identified as certain analytic Lipschitz spaces (see \cite[Theorem 5.1]{Duren-Hp}). This distinction justifies the following result.

\begin{corollary}
Let $\Omega \subset \C$ be a domain. Then, $\bloch(\Omega) = 1 + \g(\Omega)$.
\begin{proof}
Let $0 < \alpha < 1$. By Lemma \ref{lemma:bloch-no-domains}, we have that $\mathrm{Hol}(\D,\Omega) \not\subset \mathcal{B}_{v_{\alpha}}$. This means that $\bloch(\Omega) \geq 1$. But then, the result follows at once from Theorem \ref{thm:Lusky}, since $\mathcal{B}_{v_{\alpha}} = \mathrm{H}^{\infty}_{v_{\alpha-1}}$.
\end{proof}
\end{corollary}

This means that Theorem \ref{thm:growth-number}, Theorem \ref{thm:examples}, Lemma \ref{lemma:growth-bounded-complement}, and Lemma \ref{lemma:growth-examples-simpconn} holds for the spaces $\mathcal{B}_{v_{\alpha}}$, $\alpha > 0$, with obvious modifications.

\section{Some final examples regarding the Bloch space}
\label{sec:final-example}
Recall that $\mathcal{B} \subset \mathrm{H}^{\infty}_v$ for the logarithmic weight $v$ given by
\begin{equation}
\label{eq:v-log}
v(z) = \left[1+\log\left(\dfrac{1}{1-\abs{z}}\right)\right]^{-1}, \qquad z \in \D,
\end{equation}
see e.g. \cite[Proposition 7.5.1]{Vukotic}. However, such spaces are quite different in the following sense.

\begin{proposition}
\label{prop:Bloch-Hv}
Let $v$ be as in \eqref{eq:v-log}. There exists a simply connected domain $\Omega \subset \C$ such that $\mathrm{Hol}(\D,\Omega) \subset \mathrm{H}^{\infty}_v$ but $\mathrm{Hol}(\D,\Omega) \not\subset \mathcal{B}$.
\begin{proof}
The weight $v$ is typical. By \cite[Examples 1.7.(d)]{BBT} we also see that $v$ is equivalent to an essential weight, and so $v$ is also essential. Moreover, it also satisfies \eqref{eq:continuity-wtv}. Therefore, it follows from Theorem \ref{thm:univconmap-Hinftyv} that a hyperbolic domain $\Omega \subset \C$ with $0 \in \Omega$ satisfies $\mathrm{Hol}(\D,\Omega) \subset \mathrm{H}^{\infty}_v$ if and only if
\begin{equation}
\label{eq:sup-c1}
\sup_{w \in \Omega}\dfrac{\abs{w}}{1+\log\left(\dfrac{1}{2}\left(e^{d_{\Omega}(0,w)}+1\right)\right)} < +\infty.
\end{equation}
The function in \eqref{eq:sup-c1} is bounded in a neighborhood of $0$, and so \eqref{eq:sup-c1} is equivalent to 
\begin{equation}
\label{eq:sup-c2}
\sup_{\substack{w \in \Omega \\ \abs{w} > \varepsilon}}\dfrac{\abs{w}}{1+\log\left(\dfrac{1}{2}\left(e^{d_{\Omega}(0,w)}+1\right)\right)} < +\infty
\end{equation}
for some (hence, all) $\varepsilon > 0$. Also, $\lim_{x \to +\infty}\frac{1}{x}\left[1+\log\left(\frac{1}{2}\left(e^{x}+1\right)\right)\right] = 1$. Then, \eqref{eq:sup-c2} is equivalent to
\begin{equation}
\label{eq:condition-B-log}
\sup_{\substack{w \in \Omega \\ \abs{w} > \varepsilon}} \dfrac{\abs{w}}{d_{\Omega}(0,w)} < + \infty,
\end{equation}
for some (hence, all) $\varepsilon > 0$. Therefore, using Theorem \ref{thm:classical-Bloch}, the result follows if we can find a simply connected domain $\Omega \subsetneq \C$ with $0 \in \Omega$ satisfying \eqref{eq:condition-B-log} but containing disks of arbitrarily large radius.

To prove the existence of such a domain, let $L_n,l_n > 1$ be two non-decreasing sequences. Consider
$$S_1 = L_1, \qquad S_{n+1} = S_n + l_n + L_{n+1}, \qquad n \in \N.$$
We will work with the domain $\Omega \subset \C$ given by 
$$\Omega = \bigcup_{n = 1}^{\infty}X_n \cup \{z = x+iy \in \C : x > -1/2, \, \abs{y} < 1/2\},$$
where
$$X_n = \{z = x+iy \in \C : S_n < x < S_n+l_n, \, \abs{y} < l_n/2\}, \qquad n \in \N.$$
The construction is illustrated in Figure \ref{fig:domain-b-log} for clarity. Notice that $\Omega$ is the union of squares defined by the variables $L_n$ and $l_n$ and a half-strip. If we impose that $l_n \to +\infty$, as $n \to \infty$, then by Theorem \ref{thm:classical-Bloch} we have that $\mathrm{Hol}(\D,\Omega) \not\subset \mathcal{B}$. We claim that we can find $L_n,l_n > 0$ with $\lim_{n \to \infty}l_n = +\infty$ so that \eqref{eq:condition-B-log} holds.

\begin{figure}[h]
\centering
\begin{tikzpicture}[scale=1]
\newcommand{\x}{{0,1,2,3.5,5,7.5,9.5,10}}
\newcommand{\w}{{0.25,0.25,0.25,0.25}}
\fstbox{\x[0]}{\w[0]}
\fill (0,0) circle (1pt);
\xbanda{\x[0]}{\x[1]}{\w[0]}
\xbox{\x[1]}{\x[2]}{\w[0]}{\w[1]}
\xbanda{\x[2]}{\x[3]}{\w[1]}
\xbox{\x[3]}{\x[4]}{\w[1]}{\w[2]}
\xbanda{\x[4]}{\x[5]}{\w[2]}
\xbox{\x[5]}{\x[6]}{\w[2]}{\w[3]}
\xbanda{\x[6]}{\x[7]}{\w[3]}
\foreach \k in {0,...,2}{
\draw[dashed] (\x[2*\k],---\w[\k]) -- (\x[2*\k],-2);
\draw[dashed] (\x[2*\k+1],---\w[\k]) -- (\x[2*\k+1],-2);
}
\draw[dashed] (\x[6],---\w[3]) -- (\x[6],-2);
\foreach \k in {0,...,5}{
\draw[<->] (\x[\k],-2) -- (\x[\k+1],-2);
}
\foreach \k in {1,2,3}{ 
\node at ({\x[\the\numexpr\k+\k-2\relax]/2+\x[\the\numexpr\k+\k-1\relax]/2},-2.5) {$L_{\k}$};
}
\foreach \k in {1,2,3}{ 
\node at ({\x[\the\numexpr\k+\k-1\relax]/2+\x[\the\numexpr\k+\k\relax]/2},-2.5) {$l_{\k}$};
}

\end{tikzpicture}
\caption{The domain $\Omega$ for some choice of $L_n,l_n > 0$. The point inside the left-most rectangle represents the origin.}
\label{fig:domain-b-log}
\end{figure}

If we impose that $l_n \leq L_n$ for all $n \in \N$, then $\abs{\arg(z)} \leq \pi/4$ for all $z \in \Omega$ with $\mathrm{Re}(z) > L_1$. Therefore, \eqref{eq:condition-B-log} is equivalent to
$$\sup_{\substack{w \in \Omega \\ \mathrm{Re}(z) > L_1}} \dfrac{\mathrm{Re}(w)}{d_{\Omega}(0,w)} < + \infty.$$
We will show that, under some conditions for $L_n$ and $l_n$, the latter condition holds. For this purpose, we will simply use that $\lambda_{\Omega}(z) \geq 1$ for all $z \in \Omega \setminus \bigcup_{n \in \N} X_n$ (see Theorem \ref{thm:estimate-density-simpcon} and notice that $\mathrm{dist}(z,\partial\Omega) \leq 1/2$). In order to do so, we proceed by induction.

\textit{Initial step}. Let $z \in \Omega$ with $0 \leq \mathrm{Re}(z) \leq S_1$. Then,
$$d_{\Omega}(0,z) = \int_{\gamma}\lambda_{\Omega}(w) \, |dw| \geq \abs{z} \geq \mathrm{Re}(z),$$
where $\gamma$ is a hyperbolic geodesic joining $0$ and $z$, and where we have applied that $\lambda_{\Omega}(w) \geq 1$ for every $w \in \Omega$ with $0 \leq \mathrm{Re}(w) \leq S_1$. Take $C_1 = 1$.

\textit{Induction step}. Let $N \in \N$. Assume that we have constructed $C_N \in (0,1]$ so that $d_{\Omega}(0,z) \geq C_N\mathrm{Re}(z)$ for all $z \in \Omega$ with $0 \leq \mathrm{Re}(z) \leq S_N$. We will prove that $d_{\Omega}(0,z) \geq C_{N+1}\mathrm{Re}(z)$ for all $z \in \Omega$ with $0 \leq \mathrm{Re}(z) \leq S_{N+1}$,
where
$$C_{N+1} := C_N\dfrac{S_N}{S_N+l_N}.$$
We do so in two complementary cases.

\textit{Case 1}. Let $z \in \Omega$ with $S_N \leq \mathrm{Re}(z) \leq S_N + l_N$. Notice that the hyperbolic geodesic joining $0$ and $z$ intersects $I_{S_N}$, where we have denoted $I_w = \{z \in \Omega : \mathrm{Re}(z) = \mathrm{Re}(w)\}$. Then,
$$\dfrac{d_{\Omega}(0,z)}{\mathrm{Re}(z)} \geq \dfrac{d_{\Omega}(0,I_{S_N})}{S_N+l_N} \geq C_N\dfrac{S_N}{S_N+l_N} = C_{N+1}.$$

\textit{Case 2}. Let $z \in \Omega$ with $S_N+l_N \leq \mathrm{Re}(z) \leq S_{N+1}$. Notice that the hyperbolic geodesic joining $0$ and $z$ intersects $I_{S_N+l_N}$. Then,
\begin{align*}
\dfrac{d_{\Omega}(0,z)}{\mathrm{Re}(z)} & \geq \dfrac{d_{\Omega}(0,I_{S_N+l_N})+d_{\Omega}(I_{S_N+l_N},I_z)}{\mathrm{Re}(z)} \\
& \geq \dfrac{C_{N+1}(S_N+l_N)+\mathrm{Re}(z)-S_N-l_N}{\mathrm{Re}(z)} \geq C_{N+1}
\end{align*}
where we have used Case 1, $\lambda_{\Omega}(w) \geq 1$ for every $w \in \Omega$ with $S_N+l_N \leq \mathrm{Re}(w) \leq S_{N+1}$, and $C_{N+1} \leq 1$.

All in all, we see that $d_{\Omega}(0,z) \geq C_{N+1}\mathrm{Re}(z)$ for all $z \in \Omega$ with $0 \leq \mathrm{Re}(z) \leq S_{N+1}$, where
$$C_{N+1} = C_N\dfrac{S_N}{S_N+l_N}.$$

Thanks to this inductive argument, we conclude that \eqref{eq:condition-B-log} holds if we choose $L_n,l_n > 0$ such that
$$\lim_{N \to +\infty}C_N = \prod_{n = 1}^{\infty}\dfrac{S_n}{S_n+l_n} = \prod_{n=1}^{\infty}\left(1-\dfrac{l_n}{l_n+S_n}\right) > 0.$$

But, for instance, one could put
$$l_n = n+1, \qquad L_n = (n+1)n^2, \qquad n \in \N.$$
In that case,
$$\sum_{n = 1}^{\infty}\dfrac{l_n}{l_n+S_n} \leq \sum_{n = 1}^{\infty}\dfrac{l_n}{l_n+L_n} = \sum_{n = 1}^{\infty}\dfrac{1}{n^2+1} < +\infty,$$
and so \eqref{eq:condition-B-log} holds.
\end{proof}
\end{proposition}

In fact, varying the height of the rectangles in Figure \ref{fig:domain-b-log}, we can prove the following general result.
\begin{proposition}
Let $v$ be a typical weight for which all composition operators are bounded on $\mathrm{H}^{\infty}_v$. Then, there exists a simply connected domain $\Omega \subset \C$ such that $\mathrm{Hol}(\D,\Omega) \subset \mathrm{H}^{\infty}_v$ but $\mathrm{Hol}(\D,\Omega) \not\subset \mathcal{B}$.
\begin{proof}
First of all, by choosing an equivalent weight, we can always assume that $v$ is strictly decreasing in $[0,1)$ and that $v \in \mathcal{C}^1([0,1))$.

Now, let $\{L_n\}$ and $\{l_n\}$ be non-decreasing sequences in $(1,+\infty)$, and let $\{h_n\}$ be a non-increasing sequence in $(0,1/2]$. Consider
$$S_1 = L_1, \qquad S_{n+1} = S_n + l_n + L_{n+1}, \qquad n \in \N.$$
We will work with the domain $\Omega \subset \C$ given by

$$\Omega = \bigcup_{n = 0}^{\infty}(X_n \cup Y_n),$$
where
$$X_n = \{z = x+iy \in \C : S_n < x < S_n+l_n, \, \abs{y} < l_n/2\}, \qquad n \in \N,$$
$$Y_1 = \{z = x+iy \in \C : -1/2 < x \leq S_1, \, \abs{y} < h_1/2\},$$
and
$$Y_{n+1} = \{z = x+iy \in \C : S_n+l_n \leq x \leq S_{n+1}, \, \abs{y} < h_{n+1}/2\}, \qquad n \in \N.$$

We refer to Figure \ref{fig:domain-B-Hinftyv} to ease the exposition of the construction. Notice that $\Omega$ is the union of squares and rectangles defined by the original variables $L_n$, $l_n$ and $h_n$.

\begin{figure}[h]
\centering
\begin{tikzpicture}[scale=1]
\newcommand{\x}{{0,1,2,3.5,5,7.5,9.5,10}}
\newcommand{\w}{{0.25,0.17,0.1,0.05}}
\fstbox{\x[0]}{\w[0]}
\fill (0,0) circle (1pt);
\xbanda{\x[0]}{\x[1]}{\w[0]}
\xbox{\x[1]}{\x[2]}{\w[0]}{\w[1]}
\xbanda{\x[2]}{\x[3]}{\w[1]}
\xbox{\x[3]}{\x[4]}{\w[1]}{\w[2]}
\xbanda{\x[4]}{\x[5]}{\w[2]}
\xbox{\x[5]}{\x[6]}{\w[2]}{\w[3]}
\xbanda{\x[6]}{\x[7]}{\w[3]}
\foreach \k in {0,...,2}{
\draw[dashed] (\x[2*\k],---\w[\k]) -- (\x[2*\k],-2);
\draw[dashed] (\x[2*\k+1],---\w[\k]) -- (\x[2*\k+1],-2);
}
\draw[dashed] (\x[6],---\w[3]) -- (\x[6],-2);
\foreach \k in {0,...,5}{
\draw[<->] (\x[\k],-2) -- (\x[\k+1],-2);
}
\foreach \k in {1,2,3}{ 
\node at ({\x[\the\numexpr\k+\k-2\relax]/2+\x[\the\numexpr\k+\k-1\relax]/2},-2.5) {$L_{\k}$};
}
\foreach \k in {1,2,3}{ 
\node at ({\x[\the\numexpr\k+\k-1\relax]/2+\x[\the\numexpr\k+\k\relax]/2},-2.5) {$l_{\k}$};
}
\end{tikzpicture}
\caption{The domain $\Omega$ for some choice of $L_n,l_n,h_n > 0$. The point inside the left-most rectangle represents the origin.}
\label{fig:domain-B-Hinftyv}
\end{figure}

As in Proposition \ref{prop:Bloch-Hv}, by Theorems \ref{thm:univconmap-Hinftyv} and \ref{thm:classical-Bloch}, it is enough to find $C,l_n,L_n,h_n > 0$ so that $l_n \to +\infty$, as $n \to \infty$, and
\begin{equation}
\label{eq:cond-B-Hinftyv-box}
d_{\Omega}(0,w) \geq \phi(C\abs{w}), \qquad w \in \Omega,
\end{equation}
where $\phi = u^{-1}$ is the inverse function of $u \colon (0,+\infty) \to (0,+\infty)$, which is given by
$$u(d)\,v\bigl({\textstyle\frac{e^d-1}{e^d+1}}\bigr) = 1, \qquad d > 0.$$
It will be useful to notice that, using the properties of $v$, we have that $\phi$ is increasing and $\phi \in \mathcal{C}^1((0,+\infty))$.

If we choose $l_n \leq L_n$ for all $n \in \N$, then $\abs{\arg(z)} \leq \pi/4$ for all $z \in \Omega$ with $\mathrm{Re}(z) > L_1$. In particular, we see that \eqref{eq:cond-B-Hinftyv-box} holds if and only if there exists $C' > 0$ such that
\begin{equation}
\label{eq:cond-B-Hinftyv-box-2}
d_{\Omega}(0,w) \geq \phi(C' \, \mathrm{Re}(w)), \qquad w \in \Omega, \, \mathrm{Re}(w) > L_1.
\end{equation}
Under some conditions for $l_n$, $L_n$, and $h_n$, we will prove by induction that \eqref{eq:cond-B-Hinftyv-box-2} holds.

\textit{Initial step}. Let $z \in \Omega$ with $0 \leq \mathrm{Re}(z) \leq S_1$. Since $v$ is typical, notice that
$$\mathrm{Re}(z)v\left(\dfrac{e^{d_{\Omega}(0,z)}-1}{e^{d_{\Omega}(0,z)}+1}\right) \leq S_1v(0).$$
In other words, it holds that $u(d_{\Omega}(0,z)) \geq \mathrm{Re}(z)/(S_1v(0)) \geq C_1\mathrm{Re}(z)$, where we have taken $C_1 = \min(\{1,1/(S_1v(0))\})$. That is,
$$d_{\Omega}(0,z) \geq \phi(C_1\mathrm{Re}(z)).$$

\textit{Induction step}. Let $N \in \N$. Assume that there exists $C_N > 0$ so that $d_{\Omega}(0,z) \geq \phi(C_N\mathrm{Re}(z))$ for all $z \in \Omega$ with $0 \leq \mathrm{Re}(z) \leq S_N$. We will prove that $d_{\Omega}(0,z) \geq \phi(C_{N+1}\mathrm{Re}(z))$ for all $z \in \Omega$ with $0 \leq \mathrm{Re}(z) \leq S_{N+1}$,
where
$$C_{N+1} := C_N\dfrac{S_N}{S_N+l_N}.$$
We do so in two complementary cases.

\textit{Case 1}. Let $z \in \Omega$ with $S_N \leq \mathrm{Re}(z) \leq S_N + l_N$. Notice that the hyperbolic geodesic joining $0$ and $z$ intersects $I_{S_N}$, where we have denoted $I_w = \{z \in \Omega : \mathrm{Re}(z) = \mathrm{Re}(w)\}$. Then,
$$d_{\Omega}(0,z) \geq d_{\Omega}(0,I_{S_N}) \geq \phi(C_NS_N) \geq \phi(C_{N+1}\mathrm{Re}(z)).$$

\textit{Case 2}. Let $z \in \Omega$ with $S_N+l_N \leq \mathrm{Re}(z) \leq S_{N+1}$. Notice that the hyperbolic geodesic joining $0$ and $z$ intersects $I_{S_N+l_N}$. Then,
\begin{align*}
d_{\Omega}(0,z) & \geq d_{\Omega}(0,I_{S_N+l_N})+d_{\Omega}(I_{S_N+l_N},z) \\
& \geq \phi(C_{N+1}(S_N+l_N)) + \dfrac{1}{h_{N+1}}(\mathrm{Re}(z)-S_N-l_N).
\end{align*}
where we have used Case 1 and that $\lambda_{\Omega}(w) \geq 1/h_{N+1}$ for every $w \in \Omega$ with $S_N+l_N \leq \mathrm{Re}(w) \leq S_{N+1}$ (see Theorem \ref{thm:estimate-density-simpcon} and notice that $\mathrm{dist}(w,\partial\Omega) \leq h_{N+1}/2$ in this case). In particular, using the regularity of $\phi$, we see that
\begin{align*}
\phi(C_{N+1}\mathrm{Re}(z)) - \phi(C_{N+1}(S_N+l_N)) & = \int_{C_{N+1}(S_N+l_N)}^{C_{N+1}\mathrm{Re}(z)}\phi'(t)dt \\
& \leq M_NC_{N+1}(\mathrm{Re}(z)-S_N-l_N),
\end{align*}
where
$$M_N = \sup(\{\phi'(t) : C_{N+1}(S_N+l_N) \leq t \leq C_{N+1}S_{N+1}\}) < +\infty.$$
Therefore, if we take $h_{N+1} \leq 1/(M_NC_{N+1})$, we have that
$$d_{\Omega}(0,z) \geq \phi(C_{N+1}\mathrm{Re}(z))$$
for all $z \in \Omega$ with $S_N+l_N \leq \mathrm{Re}(z) \leq S_{N+1}$.

Once more, we conclude that \eqref{eq:cond-B-Hinftyv-box} holds if we choose all $L_n$ and all $l_n$ so that
$$\lim_{N \to +\infty}C_N = C_1\prod_{n = 1}^{\infty}\dfrac{S_n}{S_n+l_n} > 0,$$
and we later choose all $h_n$ so that $h_{N+1} \leq 1/(M_NC_{N+1})$. Then, we can conclude as in the proof of Proposition \ref{prop:Bloch-Hv}.
\end{proof}
\end{proposition}

\end{document}